\documentclass[11pt]{amsart}
\usepackage{hyperref}

\usepackage{enumerate}

\usepackage[utf8]{inputenc}
\usepackage[T1]{fontenc}

\usepackage[mathscr]{eucal}
\usepackage{amsmath,amsfonts,amsthm,amssymb}
\usepackage[all]{xy}

\usepackage[a4paper, hmargin={3cm,3cm}, centering]{geometry}

\newtheorem{lemma}{Lemma}[section]
\newtheorem{theorem}[lemma]{Theorem}
\newtheorem*{theorem*}{Theorem}
\newtheorem{corollary}[lemma]{Corollary}

\newtheorem*{question*}{Question}
\newtheorem{proposition}[lemma]{Proposition}
\newtheorem*{proposition*}{Proposition}

\theoremstyle{remark}

\theoremstyle{definition}
\newtheorem*{definition*}{Definition}
\newtheorem*{conjecture*}{Conjecture}
\newtheorem*{remark*}{Remark}
\newtheorem*{remarks*}{Remarks}
\newtheorem*{claim*}{Claim}

\newcommand{\C}{{\mathbb C}}
\newcommand{\E}{{\mathbb E}}

\newcommand{\N}{{\mathbb N}}

\newcommand{\R}{{\mathbb R}}
\newcommand{\T}{{\mathbb T}}
\newcommand{\Z}{{\mathbb Z}}

\newcommand{\CA}{{\mathcal A}}

\newcommand{\CC}{{\mathcal C}}
\newcommand{\CF}{{\mathcal F}}

\newcommand{\CT}{{\mathcal T}}

\newcommand{\norm}[1]{\left\Vert #1\right\Vert}
\newcommand{\nnorm}[1]{\lvert\!|\!| #1|\!|\!\rvert}

\begin{document}

\title[Multiple ergodic averages for tempered functions]{Multiple ergodic averages for tempered functions}

\author{Andreas Koutsogiannis}

\address[Andreas Koutsogiannis]{The Ohio State University, Department of Mathematics, Columbus, Ohio, USA} \email{koutsogiannis.1@osu.edu}

\begin{abstract}
Following Frantzikinakis' approach on averages for Hardy field functions of different growth, we add to the topic by studying the corresponding averages for tempered functions, a class which also contains functions that oscillate and is in general more restrictive to deal with. Our main result is the existence and the explicit expression of the $L^2$-norm limit of the aforementioned averages, which turns out, as in the Hardy field case, to be the ``expected'' one. The main ingredients are the use of, the now classical, PET induction (introduced by Bergelson), covering a more general case, namely a ``nice'' class of tempered functions (developed by Chu-Frantzikinakis-Host for polynomials and Frantzikinakis for Hardy field functions) and some equidistribution results on nilmanifolds (analogous to the ones of Frantzikinakis' for the Hardy field case).
\end{abstract}



\maketitle

\section{Introduction}




In 1977, Furstenberg provided (in \cite{Fu}) a purely ergodic theoretical proof of Szemer\'{e}di's theorem, i.e., every subset of natural numbers with positive upper density contains arbitrarily long arithmetic progressions,\footnote{ For a subset $A\subseteq \mathbb{N}$ we define its \emph{upper density} to be the number $\limsup_{N\to\infty}\frac{|A\cap\{1,\ldots,N\}|}{N}.$ The \emph{lower density} is defined analogously with the use of $\liminf$ in place of $\limsup$; when these two values coincide, the common value is called \emph{density} of $A$.} by studying the $L^2$-norm behavior as $N\to\infty$ of the multiple ergodic averages:
\begin{equation}\label{E:Furstenberg}
\frac{1}{N}\sum_{n=1}^N T^{n}f_1\cdot T^{2n}f_2\cdots T^{\ell n}f_\ell,
\end{equation}
where $T:X\to X$ is an invertible measure preserving transformation on the probability space $(X,\mathcal{B},\mu)$ \footnote{ We call the quadruple $(X,\mathcal{B},\mu,T)$ \emph{system}.} and $f_1,\ldots,f_\ell\in L^\infty(\mu).$

More specifically, he showed that for any $A\in\mathcal{B}$ with $\mu(A)>0$ we have that \begin{equation}\label{E:Furstenberg2}
 \liminf_{N\to\infty}\frac{1}{N}\sum_{n=1}^N \mu(A\cap T^{-n}A\cap T^{-2n}A\cap\ldots\cap T^{-\ell n}A)>0
 \end{equation} and, via a Correspondence Principle (see \cite{Fu}), obtained the aforementioned result.
 
  It is worth mentioning that the existence of the limit in \eqref{E:Furstenberg} (and \eqref{E:Furstenberg2}) was not known at the time; relatively recently (in \cite{HK99}) Host and Kra not only proved its existence but actually provided an explicit expression of it.
  
In the same paper, \cite{Fu}, under the weakly mixing assumption of $T,$\footnote{ $T$ (and the corresponding system $(X,\mathcal{B},\mu,T)$) is called \emph{weakly mixing} (w.m.),  if  $T\times T$ is \emph{ergodic} for $\mu\times\mu$ (i.e., the only $T\times T$-invariant measurable sets are the ones of trivial measure in $\{0,1\}$).  } Furstenberg showed that for any $\ell\in \mathbb{N}$ and $f_1,\ldots, f_\ell\in L^\infty(\mu)$ we have 
\begin{equation}\label{E:Furstenberg3}
\frac{1}{N}\sum_{n=1}^N T^n f_1\cdot T^{2n} f_2 \cdots T^{\ell n} f_\ell\to \prod_{i=1}^\ell \int f_i\; d\mu,
\end{equation}
as $N\to\infty,$ where the convergence takes place in $L^2(\mu).$
 
\medskip 
 
From now on, except otherwise stated, every limit that we take is understood to be an $L^2$-norm limit as $N\to\infty$; we will also refer to the quantity $\prod_{i=1}^\ell \int f_i$ as the \emph{expected limit}.

\subsection{The polynomial case}
The first far-reaching extension of Furstenberg's w.m. convergence result came ten years later (in \cite{Be87a}).
Bergelson was the first to view the iterates $n, 2n, \ldots, \ell n$ as linear polynomials $p_1, \ldots, p_\ell$ with the property $p_i-p_j\neq$ constant for all $i\neq j.$
We call the non-constant polynomials $p_1,\ldots,p_\ell \in \mathbb{R}[t]$ \emph{essentially distinct} if $p_i-p_j\neq$ constant for all $i\neq j.$ We also call \emph{integer polynomials} the polynomials that take integer values at integers.


Exploiting the van der Corput trick (see Subsection~\ref{vdC,FCP} below), introducing  the PET induction (i.e., \emph{Polynomial Exhaustion Technique}), the only method that we have to this day in reducing the complexity to deal with polynomial, Hardy field or tempered iterates, Bergelson showed that, when $T$ is weakly mixing, $a_i$'s are essentially distinct integer polynomials and $f_1,\ldots, f_\ell\in L^\infty(\mu)$, 
\begin{equation}\label{E:Single}
\frac{1}{N}\sum_{n=1}^N T^{a_1(n)}f_1\cdots T^{a_\ell(n)}f_\ell
\end{equation}
converges to the expected limit. Some years later, extending Furstenberg's method, Bergelson and Leibman (in \cite{BL0}) established polynomial extensions of Szemer{\'e}di's theorem for $a_i$ integer polynomials with $a_i(0)=0,$ $1\leq i\leq \ell.$



Switching to multiple transformations, we call $(X,\mathcal{B},\mu,T_1,\ldots,T_\ell)$ \emph{system} if each $T_i$ is an invertible measure preserving transformation and $T_iT_j=T_jT_i$ for all $i,j.$

When we deal with multiple $T_i$'s we are interested in the convergence of the expression 
\begin{equation}\label{E:Multiple}
\frac{1}{N}\sum_{n=1}^N T_1^{a_1(n)}f_1\cdots T_\ell^{a_\ell(n)}f_\ell,
\end{equation}
where, once more,  $(a_i(n))_n,$ $1\leq i\leq \ell,$ are appropriate, integer valued, sequences and $f_i$'s are bounded functions. In this setting the picture is totally different. Even for essential distinct polynomials, the assumptions on $T_i$'s are not clear, not even in the special case where $\ell=2$ and $a_1(t)=t^2+t,$ $a_2(t)=t^2$ (for more details on general polynomial $\mathbb{Z}^d$-actions see \cite{DKS2}).  Chu, Frantzikinakis and Host though showed (in \cite{CFH}) that in the case where each $T_i$ is weakly mixing and $a_i$'s are non-constant integer polynomials of distinct degrees, \eqref{E:Multiple} converges to the expected limit (for the corresponding result with iterates $[a_i(n)],$ where $a_i\in\mathbb{R}[t],$ see \cite{K1}).\footnote{ Note that this result, doesn't cover the case $a_1(t)=t^2+t,$ $a_2(t)=t^2$ that we mentioned before.} This result was used to show the convergence of averages with polynomial iterates of distinct positive degrees for any system.


 

A point that has to be highlighted here is that the aforementioned result is an implication of a result for products of transformations with iterates forming a ``nice'' family of polynomials. Hence, a big difference to the single transformation case ($\mathbb{Z}$-action) is that in order to have convergence of \eqref{E:Multiple} to the expected limit, one has to prove something stronger about products of transformations ($\mathbb{Z}^\ell$-actions) with ``nice'' iterates.

Solving a conjecture of Bergelson and Leibman (stated in \cite{BL0}) Walsh (in \cite{W12}) showed that the expressions in \eqref{E:Multiple} always have limit for $a_i$'s polynomials with integer values. (Actually, Walsh proved an even more general result for $\mathbb{Z}^\ell$-actions along F{\o}lner sequences, where the transformations generate a nilpotent group--for the corresponding result with commuting transformations and iterates $[a_i(n)],$ where $a_i\in \mathbb{R}[t],$ see \cite{K1}.)

\subsection{Hardy field and tempered functions}

In this subsection we define two important and more exotic, comparing to the polynomial ones, classes of functions; the Hardy field and the tempered ones. 

\subsection*{Hardy field functions}  Let $B$ be the collection of equivalence classes of real valued functions defined on some halfline $(c,\infty),$ $c\geq 0,$ where two functions that agree eventually are identified. These equivalence classes are called \emph{germs} of functions.  A \emph{Hardy field} is a subfield of the ring $(B, +, \cdot)$ that is closed under differentiation.\footnote{We use the word \emph{function} when we refer to elements of $B$ (understanding that all the operations defined and statements made for elements of $B$ are considered only for sufficiently large values of $x\in \mathbb{R}$).} 

\medskip

Usually, one deals with Hardy field functions $g$ of \emph{polynomial growth}, i.e., 
functions which are growing strictly faster than $x^{i}\log x$ and strictly slower than $x^{i+1}$ (called of \emph{polynomial degree} $i$), and more specifically with the class of logarithmico-exponential Hardy field functions, $\mathcal{LE}$,\footnote{ $a$ is a \emph{logarithmico-exponential Hardy field function} if it belongs to a Hardy field of real valued functions and it's defined on some $(c,+\infty),$ $c\geq 0,$ by a finite combination of symbols $+, -, \times, \div, \sqrt[n]{\cdot}, \exp, \log$ acting on the real variable $x$ and on real constants (for more on Hardy field functions and in particular for logarithmico-exponential ones one can check \cite{F2, F3}).} which can be handled more easily. 

\medskip

A more restrictive to the previous class to work with, is that of tempered functions.

\subsection*{Tempered functions} Let $i$ be a non-negative integer. A real-valued function $g$ which is $(i+1)$-times continuously differentiable on $[x_0,\infty),$ where $x_0\geq 0,$ is called a \emph{tempered function of degree $i$} (we write $\deg g=i$), if the following hold:
\begin{enumerate}
\item[(1)] $g^{(i+1)}(x)$ tends monotonically to $0$ as $x\to\infty;$

\item[(2)]  $\lim_{x\to\infty}x|g^{(i+1)}(x)|=\infty.$
\end{enumerate}
Tempered functions of degree $0$ are called \emph{Fej\'{e}r functions}. (See \cite{BK} for more details on tempered functions.)

\medskip

For a weakly mixing transformation, Bergelson and H{\aa}land-Knutson showed that \eqref{E:Single} converges to the expected limit for $a_i(n)=[g_i(n)],$ $1\leq i\leq \ell,$  where $g_i\in \mathcal{G}\cap\mathcal{LE}$ with $g_i-g_j\in \mathcal{G}$ for $i\neq j$ (\cite[Theorem A]{BK});\footnote{ $\mathcal{G}=\mathcal{T}\cup\mathcal{P},$ where $\mathcal{P}=\bigcup_{i\geq 0}\{g\in C^{\infty}(\mathbb{R}^+):\;\exists\;\gamma\in \mathbb{R}\setminus\{0\},\;\lim_{x\to\infty}g^{(i+1)}(x)=\gamma,$ $ \;\lim_{x\to\infty} x^jg^{(i+j+1)}(x)=0,\;j\in\mathbb{N}\}$ contains the real polynomials of positive degree and $\mathcal{T}$ is a special subclass of tempered functions that we will use throughout this article and is defined in the next section.} or $g_i\in \mathcal{G}$ with $g_i-g_j\in \mathcal{G}$ for $i\neq j$ and such that the family $\{g_1,\ldots,g_\ell\}$ has the $R$-property (see \cite[Definition~1.10]{BK}).

Frantzikinakis was the first one to obtain convergence results providing the precise expression of the limit for general systems, i.e., under no assumption(s) on the transformation(s). Verifying a conjecture from \cite{BK}, he proved (in \cite[Theorem~2.6]{F2}) that for  $a_i(n)=[g_i(n)],$ where $g_\ell\prec\ldots\prec g_1\in \mathcal{LE}\cap \mathcal{U},$ \footnote{ $\mathcal{U}=\{g\in C(\mathbb{R}^+):\;x^{k+\varepsilon}\prec g(x)\prec x^{k+1}$ for some $k\in \omega$ and some $\varepsilon>0\},$ where $\omega$ denotes the set of whole numbers, i.e., non-negative integers. We write $g_2\prec g_1$ if $|g_1(x)|/|g_2(x)|\to \infty$ as $x\to\infty.$}  \eqref{E:Single} converges to the expected limit.\footnote{ Which for $T$ ergodic is the product of integrals of $f_i$'s, while for a general $T$ is the product of conditional expectations of $f_i$'s with respect to the $\sigma$-algebra of $T$-invariant sets.}

As we already mentioned, for averages with multiple w.m. $T_i$'s, to study \eqref{E:Multiple} and show convergence to the expected limit, one should first show a stronger result about ``nice'' subfamilies of the families of functions of interest.
Following the polynomial setting (\cite{CFH}), Frantzikinakis defined what a nice family of Hardy field functions is and showed 
 (in \cite[Theorem~2.3]{F3}) that \eqref{E:Multiple}, with $a_i(n)=[g_i(n)],$ where $g_\ell\prec\ldots\prec g_1\in \mathcal{H}\cap \mathcal{V},$\footnote{ $\mathcal{H}$ is a Hardy field and $\mathcal{V}=\{g\in C(\mathbb{R}^+):\;x^{k}\log x\prec g(x)\prec x^{k+1}$ for some $k\in \omega\}$.} also converges to the expected limit (i.e., to the product of the corresponding conditional expectations).

The generality of these results, i.e., their validity in any system, allows one to get various recurrence and combinatorial results obtaining interesting patterns on subsets of integers with positive (upper) density. Frantzikinakis, by applying equidistribution results on nilmanifolds for the corresponding appropriate classes of functions, showed that the limit of \eqref{E:Single} is equal to the one of \eqref{E:Furstenberg}. Then, by using Furstenberg's Correspondence Principle he obtained refinements of Szemer{\'e}di's theorem. More specifically, he did that for $a_i(x)=i[p(x)],$ where $p\in \mathbb{R}[x]$ is a real valued polynomial with  $p(x)\neq cq(x)+d,$ for all $c,d\in \mathbb{R}$ and $q\in \mathbb{Q}[x]$ (this follows by the proof of \cite[Theorem~2.2]{F2}--for a convergence result for general systems and single $T$ with iterates ``strongly independent polynomials'', see \cite{KK}); for $a_i(x)=i[g(x)],$ where $g$ is a Hardy field function of polynomial growth satisfying $\log x\prec |g(x)-cp(x)|$ for every $c\in \mathbb{R}$ and $p\in \mathbb{Z}[x]$ (\cite[Theorem~2.2]{F2}) and, finally, for $a_i(x)=i[g(x)],$ where $g$ is a tempered function from the class $\mathcal{T}$.\footnote{ This last claim follows from the results of \cite{F2} and the fact that a tempered function $g$ from $\mathcal{T}$ satisfies: $|g^{(k+1)}(x)|$ decreases to $0,$ $1/x^k\prec g^{(k)}(x)\prec 1$ and $(g^{(k+1)}(x))^k\prec (g^{(k)}(x))^{k+1}$ for some $k\in\mathbb{N}$ (see also Proposition~\ref{P:main_qaui} below).}

A result that is missing from the picture, and it is one among other general ones for ergodic averages that we are dealing with in this article, is to show that \eqref{E:Multiple}, for general systems, converges to the expected limit for tempered functions of different growth rates; adding one more class to the very short list of families of functions for which we have knowledge of such a limiting behavior. We are doing this by following the corresponding approach of the aforementioned results in the multiple transformations setting. This is a complementary work to the one of Frantzikinakis' from \cite{F2, F3, F4}, covering also some additional, 
for a single w.m. transformation, to \cite{BK} cases.
 
\subsection*{Notation} With $\N=\{1,2,\ldots\},$ $\omega=\N\cup\{0\},$ $\mathbb{Z},$ $\mathbb{Q},$ and $\mathbb{R}$ we denote the set of natural, whole, integer, rational and real numbers respectively. For a measurable function $f$ on a measure space $X$
with a transformation $T:X\to X,$ we denote with $Tf$ the composition $f\circ T.$  For $s\in \N,$ $\T^s =\R^s/\Z^s $ denotes the $s$ dimensional torus, and $(a(n))_n$ denotes a sequence indexed over the natural numbers (i.e., $(a(n))_{n\in \mathbb{N}}$).



\section{Main results}\label{results}

In this section we define the special classes of tempered functions that we work with (see also \cite{BK}) and state the main results of this article.

\begin{definition*}
Let $\mathcal{R}:=\Big\{g\in C^\infty(\mathbb{R}^+):\;\lim_{x\to\infty}\frac{xg^{(j+1)}(x)}{g^{(j)}(x)}\in \mathbb{R}\;\;\text{for all}\;\;j\in\omega\Big\};$

$\mathcal{F}:=\Big\{g\in C^\infty(\mathbb{R}^+):\;g\;\;\text{is Fej\'{e}r and}\;\;\exists\;\alpha\in (0,1]\;\;\text{such that}\;\;\lim_{x\to\infty}\frac{xg''(x)}{g'(x)}=\alpha-1\Big\};$

$\mathcal{T}_i:=\Big\{g\in\mathcal{R}:\;\exists\;i<\alpha\leq i+1,\;\lim_{x\to\infty}\frac{xg'(x)}{g(x)}=\alpha,\;\lim_{x\to\infty}g^{(i+1)}(x)=0\Big\};$ 

and $\mathcal{T}:=\bigcup_{i=0}^\infty \mathcal{T}_i.$
\end{definition*}




We will mainly work with the class of functions $\mathcal{T}.$ 
It is a known fact that every element of $\CT_i$ 
is a tempered function of degree $i$ (see \cite{BK}). Note that a big difference between Hardy field functions (where limits of ratios always exist) and tempered functions from $\CT$ is that in the latter case, since 
\[\frac{g_2^{(k+1)}(x)}{g_1^{(k+1)}(x)}=\frac{x g_2^{(k+1)}(x)}{g_2^{(k)}(x)}\cdot \frac{g_1^{(k)}(x)}{xg_1^{(k+1)}(x)} \cdot \frac{g_2^{(k)}(x)}{g_1^{(k)}(x)},\]
we might not be able to compare growth rates of derivatives of functions (the limit of $x g_1^{(k+1)}(x)/g_1^{(k)}(x)$ may be $0$). This will prevent us from having, even for functions of different growth rates from $\CT$, that non-trivial linear combinations of them are still in $\CT.$\footnote{ This is the main reason why in \cite{BK}, when dealing with Fej{\'e}r functions, the authors assume that the ratios are eventually monotone functions.} About the growth rates, for $g_1,$ $g_2$ real valued functions defined on some half-line $[x_0,\infty)$\footnote{ We denote with $\mathbb{R}^+$ any such half-line.} we recall that we write $g_2\prec g_1$ if $|g_1(x)|/|g_2(x)|\to \infty$ (equiv., $g_1(x)/g_2(x)\to 0$) as $x\to \infty.$

\medskip

In Section~\ref{vN}, exploiting the uniform distribution (or equidistribution) properties of tempered functions, we prove the following von Neumann-type result: 

\begin{theorem}\label{T:vNt}
For $\ell\in \mathbb{N}$ let $(X,\mathcal{B},\mu,T_1,\ldots,T_\ell)$ be a system and $f\in L^2(\mu).$ If $g_\ell\prec \ldots\prec g_1\in\mathcal{T}$ are such that $\sum_{i=1}^{\ell}\lambda_i g_i\in \CT$ for all $(\lambda_1,\ldots,\lambda_\ell)\in \mathbb{R}^\ell\setminus\{\vec{0}\},$\footnote{ Note that this condition is equivalent in saying that any non trivial linear combination of the $g_i$'s is still in $\mathcal{R}$ (this claim follows from III (iv) of the next section).} then \[\lim_{N\to\infty}\norm{\frac{1}{N}\sum_{n=1}^N T_1^{[g_1(n)]}\cdots T_\ell^{[g_\ell(n)]}f-Pf}_2=0,\footnote{ $[\cdot]$ is the integer part function, i.e., for $x\in \mathbb{R},$ $[x]$ is the largest integer which is $\leq x.$}\] where $P$ denotes the projection on the set $\{f\in L^2(\mu):\;T_i f=f$ for all $1\leq i\leq\ell\}.$
\end{theorem}

Note here that if the $\mathbb{Z}^\ell$-action $T$ defined by $T^{(n_1,\ldots,n_\ell)}:=T_1^{n_1}\cdots T_\ell^{n_\ell}$ is ergodic, then the previous limit becomes the expected one as it equals to $\int f\;d\mu.$ 

Letting $f={\bf{1}}_A,$ using the relation 
\[\langle f, Pf\rangle=\langle f, P^2f\rangle=\langle Pf, Pf\rangle\geq (\mu(A))^2,\] we get the following corollary of Theorem~\ref{T:vNt}.

\begin{corollary}\label{C:sub1}
Under the assumptions of Theorem~\ref{T:vNt}, for every $A\in \mathcal{B}$ we have:
\[\lim_{N\to\infty}\frac{1}{N}\sum_{n=1}^N \mu(A\cap T_1^{-[g_1(n)]}\cdots T_\ell^{-[g_\ell(n)]}A)\geq (\mu(A))^2.\]
\end{corollary}

Using Furstenberg's correspondence principle,\footnote{ We are actually using here the following version of it: given $E\subseteq \mathbb{Z}^\ell$ there is a system $(X,\mathcal{B},\mu,T_1,\ldots,T_\ell)$ and $A\in \mathcal{B}$ with $\mu(A)=d^\ast(E):=\sup_{\{I_n\}_n}\limsup_{n\to\infty}(|E\cap I_n|/|I_n|),$ where the supremun is taken along all the parallelepipeds with $|I_n|\to \infty,$ such that for any $(n_1,\ldots,n_\ell)\in \mathbb{Z}^\ell$ we have  $d^\ast(E\cap(E-(n_1,\ldots,n_\ell)))\geq\mu(A\cap T_1^{-n_1}\cdots T_\ell^{-n_\ell}A).$} Corollary~\ref{C:sub1} implies:

\begin{corollary}\label{C:sub2}
Let $g_1,\ldots,g_\ell$ be as in Theorem~\ref{T:vNt}. Then for every $E\subseteq \mathbb{Z}^\ell$ we have
\[\liminf_{N\to\infty}\frac{1}{N}\sum_{n=1}^N d^\ast(E\cap(E-([g_1(n)],\ldots,[g_\ell(n)])))\geq (d^\ast(E))^2.\]
\end{corollary}

We remark at this point that someone (as in \cite[Theorem~4.3 and Corollary~4.2]{BKS} but for averages along $\N$) can have the corresponding to Corollaries~\ref{C:sub1} and \ref{C:sub2} results for $m$-tuples $(\psi_1,\ldots,\psi_m)=L([g_1],\ldots,[g_\ell]),$ where $L:\mathbb{Z}^\ell\to\mathbb{Z}^m$ is a linear transformation.

\medskip

In Section~\ref{slF} we deal with subclasses of Fej{\'e}r functions.
We show the following result for functions from the class $\mathcal{F}$ which, as in the corresponding Hardy field case (see \cite[Theorem~2.7]{F2}), surprisingly enough, holds without any commutativity assumption on the transformations.

\begin{theorem}\label{T:sub-linear}
For $\ell\in \mathbb{N}$ let $(X,\mathcal{B},\mu,T_i),$ $1\leq i\leq \ell,$ be systems and $f_1,\ldots, f_\ell\in L^\infty(\mu).$ If  $g_\ell\prec \ldots\prec g_1\in \CF$ with $(g'_{i+1}(x)/g'_i(x))$ eventually monotone for all $1\leq i\leq \ell-1,$ 
then we have that \[\lim_{N\to\infty}\norm{\frac{1}{N}\sum_{n=1}^N T^{[g_1(n)]}_1 f_1\cdots T^{[g_\ell(n)]}_\ell f_\ell -\prod_{i=1}^\ell\mathbb{E}(f_i|\mathcal{I}(T_i))}_2=0.\footnote{  $\mathbb{E}(f_i|\mathcal{I}(T_i))$ denotes the conditional expectation of $f_i$ with respect to the $\sigma$-algebra $\mathcal{I}(T_i)$ of the $T_i$-invariant sets (see also Subsection~\ref{Ss:semi}). We remark that the conclusion of Theorem~\ref{T:sub-linear} holds also for functions $g_\ell\prec \ldots\prec g_1\in\mathcal{T}_0$, without any additional assumption on the ratios of the derivatives (see remark after the proof in Section~\ref{slF}).}\]
\end{theorem}

\medskip

In Section~\ref{wm}, we prove the following result for $\mathbb{Z}^\ell$-actions on nice families of tempered functions (nice families are defined in Section~\ref{S:nice} while the seminorms in Subsection~\ref{Ss:semi}):

\begin{proposition}\label{P:1}
For $\ell, m\in \mathbb{N}$ let $(X,\mathcal{B},\mu,T_1,\ldots,T_\ell)$ be a system, $f_1,\ldots,f_m\in L^\infty(\mu)$  and $((g_{1,1},\ldots,g_{\ell,1}),\ldots,(g_{1,m},\ldots,g_{\ell,m}))$ a nice ordered family of $\ell$-tuples of functions with $\deg g_{1,1}= d\in \omega.$ There exists $k\equiv k(d,\ell,m)\in\mathbb{N}$ such that if $\nnorm{f_1}_{k,T_1}=0,$ then
\begin{equation}\label{Prop: nice}
\lim_{N\to\infty}\sup_{E\subseteq \mathbb{N}}\norm{\frac{1}{N}\sum_{n=1}^N\prod_{j=1}^m (T_1^{[g_{1,j}(n)]}\cdots T_\ell^{[g_{\ell,j}(n)]})f_j\cdot {\bf{1}}_E(n)}_2=0.
\end{equation}
\end{proposition}

This result covers the weakly mixing case:

\begin{corollary}\label{cor:wm}
For $\ell\in \mathbb{N}$ let $(X,\mathcal{B},\mu,T_1,\ldots,T_\ell)$ be a weakly mixing system,\footnote{ By this we mean that the $T_i$'s commute and are weakly mixing.} $f_1,\ldots,f_\ell\in L^\infty(\mu)$ and $g_\ell\prec\ldots\prec g_1\in \mathcal{T}.$ Then we have that \[\lim_{N\to\infty}\norm{\frac{1}{N}\sum_{n=1}^N T_1^{[g_1(n)]}f_1\cdots T_\ell^{[g_\ell(n)]}f_\ell-\prod_{i=1}^\ell \int f_i\; d\mu}_2=0.\]
\end{corollary}

We note that the previous result, even though doesn't cover all the cases of \cite[Theorem~B]{BK} for a single transformation $T$ (for example, we don't deal with polynomial iterates) it does cover some additional ones (see Section~\ref{wm} below for more details). 

\medskip

Adding (as in Theorem~\ref{T:vNt}) the assumption that the functions $g_\ell\prec\ldots\prec g_1\in \CT$ are such that any non-trivial linear combination of them is still in $\CT$, fact that will allow us to obtain equidistribution results (analogous to those from \cite{F4}) in Section~\ref{S:equi},  we get the main result of the article (which we prove in Section~\ref{main}), i.e., that the limit of the average of interest is the expected one:

\begin{theorem}\label{T:main_general}
For $\ell\in \mathbb{N}$ let $(X,\mathcal{B},\mu,T_1,\ldots,T_\ell)$ be a system and $f_1,\ldots,f_\ell\in L^\infty(\mu).$  If $g_\ell\prec \ldots\prec g_1\in\mathcal{T}$ are such that $\sum_{i=1}^{\ell}\lambda_i g_i\in \CT$ for all $(\lambda_1,\ldots,\lambda_\ell)\in \mathbb{R}^\ell\setminus\{\vec{0}\},$ then  \[\lim_{N\to\infty}\norm{\frac{1}{N}\sum_{n=1}^N T_1^{[g_1(n)]}f_1\cdots T_\ell^{[g_\ell(n)]}f_\ell-\prod_{i=1}^\ell \mathbb{E}(f_i|\mathcal{I}(T_i))}_2=0.\]
\end{theorem}

Because of the generality of Theorem~\ref{T:main_general}, being valid for all systems (as Theorems~\ref{T:vNt} and \ref{T:sub-linear}), one can easily get various recurrence, combinatorial and topological dynamical applications. While Theorem~\ref{T:main_general} implies all the corresponding applications from \cite[Section~2]{F3}, we chose to indicatively state two of them (see \cite[Corollary~2.5 and Theorem~2.6]{F3}):

\begin{corollary}\label{C:gen1}
Under the assumptions of Theorem~\ref{T:main_general}, for every $A\in \mathcal{B}$ we have:
\[\lim_{N\to\infty}\frac{1}{N}\sum_{n=1}^N \mu(A\cap T_1^{-[g_1(n)]}A\cap\cdots\cap T_\ell^{-[g_\ell(n)]}A)\geq (\mu(A))^{\ell+1}.\]
\end{corollary}

One can get, by using a variant of Furstenberg correspondence principle, an analogous to Corollary~\ref{C:gen1} combinatorial consequence of Theorem~\ref{T:main_general} (see \cite[Corollary~2.9]{F3}). We chose instead to present a result in topological dynamics (see \cite[Theorem~2.6]{F3}).

\begin{corollary}\label{C:gen2}
Let $g_1,\ldots,g_\ell$ be as in Theorem~\ref{T:main_general}, $(X,d)$ a compact metric space and $T_1,\ldots,T_\ell$ invertible, commuting, minimal transformations from $X$ to itself. Then, for a residual and $T_i$-invariant set of $x\in X$ we have
\[\overline{\{(T_1^{[g_1(n)]}x,\ldots,T_\ell^{[g_\ell(n)]}x):\;n\in \mathbb{N}\}}=X^\ell.\]
\end{corollary}

\begin{remark*}
Letting $\lceil x\rceil$ and $[[x]]$ denote the smallest integer which is $\geq x$ and the closest integer to $x$ respectively, using the relations $\lceil x\rceil=-[-x]$ and $[[x]]=[x+1/2]$,
we see that all the previous results, together with their implications, remain true if, in the expressions of interest, the $[\cdot]$'s are individually and independently replaced by any of $[\cdot],$ $\lceil \cdot \rceil,$ or $[[\cdot]].$ 
\end{remark*}

\section{Facts about tempered functions, \\ van der Corput lemma and background material}

In this section, we state some general facts about tempered functions, describe the sets that were defined in Section~\ref{results} and discuss the crucial tool of van der Corput. We also provide some information on the background material that we use throughout the paper.

\subsection*{I. General facts about tempered functions.} (i) The conditions (1) and (2) of the definition of a tempered function imply that every such function is eventually monotone (fact that is true for Hardy field functions as well).

\medskip

(ii) Any tempered function $g$ of degree $i$ satisfies the growth conditions: $x^i\log x\prec g(x)\prec x^{i+1}$ (see \cite{BK}).

\medskip

For $g_1,$ $g_2$ defined on $\mathbb{R}^+,$ we write $g_2\ll g_1$ if there exists a constant $C$ such that $|g_2(x)|\leq C|g_1(x)|$ for $x$ sufficiently large.

\medskip

(iii) One can be more specific on the growth rates of a function from $\mathcal{T}.$ We will show that every function $g\in\mathcal{T}$ with $\lim_{x\to \infty}(xg'(x)/g(x))=\alpha,$ behaves ``almost'' as $x^\alpha.$

\medskip

If $g\in \CT_{i_0},$ then for $0<\varepsilon<\alpha-i_0$ we eventually have that $x^{\alpha-\varepsilon}\ll g(x)$ (in case $\alpha<i_0+1,$ we also eventually have that $g(x)\ll x^{\alpha+\varepsilon}$ for every $0<\varepsilon<i_0+1-\alpha$).\footnote{ So, the set $\mathcal{T}$ doesn't contain functions which are slower than any power of $x$.}

Indeed, there exists $\alpha\in (i_0,i_0+1]$ with $\lim_{x\to\infty}\frac{xg'(x)}{g(x)}=\alpha,$ so for any $\varepsilon>0$ there exists $M>0$ such that $\alpha-\varepsilon<\frac{xg'(x)}{g(x)}<\alpha+\varepsilon$ for all $x>M,$ hence $$\log\left(\frac{x}{M}\right)^{\alpha-\varepsilon}=\int_M^x \frac{\alpha-\varepsilon}{t}\;dt\leq \log\frac{|g(x)|}{|g(M)|}=\int_{M}^x \frac{g'(t)}{g(t)}\;dt\leq \int_M^x \frac{\alpha+\varepsilon}{t}\;dt=\log\left(\frac{x}{M}\right)^{\alpha+\varepsilon},$$ from which the claim follows.

\medskip

(iv) By a classical result of Fej\'{e}r, we have that if $g$ is a Fej\'{e}r function, then the sequence $(g(n))_n$ is equidistributed in $\T$ (see Subsection~\ref{Ss:ud} below for the definition). Using van der Corput's difference theorem (i.e., $(x_n)_n$ is equidistributed in $\T$ if $(x_{n+h}-x_n)_n$ is equidistributed in $\T$ for any $h\in\mathbb{N}$) and the aforementioned result, we get (see \cite{BK}) that if $g$ is a tempered function, then $(g(n))_n$ is equidistributed in $\T$.\footnote{ Notice here that a Hardy field function doesn't necessarily have this property (but it does when it is ``away'' from polynomials). Also, note that the cutoff on the growth rate is crucial as $(\log n)_n$ is not equidistributed in $\T$ whereas for every $\varepsilon>0,$ $(\log^{1+\varepsilon} n)_n$ is. }

\subsection*{II. Differences between the classes of sublinear functions}

Here we list some facts about functions from the sets $\mathcal{F},$ $\mathcal{T}_0$ and $\mathcal{LE}_0.$\footnote{ We denote with $\mathcal{LE}_0$ the set of logarithmico-exponential functions, $g,$ of degree $0,$ i.e., $\log x\prec g(x)\prec x.$}

\medskip

(i) $\mathcal{F}$ is a proper subset of Fej\'{e}r functions since for example $\log^\beta x,$ $\beta>1$ is a Fej\'{e}r function that doesn't belong to $\mathcal{F}.$

\medskip

(ii) $\CT_0$ is a proper subset of $\CF,$ since $g(x)=x^{\alpha}(1+(\cos x/x^{\beta}))\in\CF\setminus \CT_0$ for $0<\alpha<1,$ $2<\beta<3,$  as the expression $xg'''(x)/g''(x)$ does not have a limit as $x\to\infty$.

\medskip

(iii)  $\mathcal{LE}_0\nsubseteq \CT_0$ since $g(x)=e^{(\log x)^\alpha}\in \mathcal{LE}_0\setminus (\CT_0\cup \CF),$ for $0<\alpha<1,$ and $\CT_0\nsubseteq \mathcal{LE}_0,$ since, from \cite{BK}, we have that $g(x)=x^{1/2} (2+\cos\sqrt{\log x})\in \CT_0\setminus \mathcal{LE}_0$ (the derivative of $g^2$ is not eventually monotone).\footnote{  Hence, our study over the family $\CT$ that includes functions with ``oscillation'', for multiple transformations, will provide us with various new results that add to those of \cite{BK}, \cite{F2}, \cite{F3} and \cite{F4}.}

\subsection*{III. Relations between the growth rates}

(i) Let $g_1, g_2\in \mathcal{T}.$ By \cite[Lemma~2.6]{BK}, assuming that $\lim_{x\to\infty}\frac{xg_i'(x)}{g_i(x)}=\alpha_i,$ $i=1,2,$ we have that if $g_2\prec g_1,$ then $\alpha_2\leq \alpha_1$. Conversely, if $\alpha_2<\alpha_1,$ then $g_2\prec g_1$ and  $g_2'/g_1'$ is eventually monotone. 

Note that it can happen $\alpha_1=\alpha_2$ while $g_2\prec g_1.$ (Indeed, one can take for example $g_1(x)=x^\alpha \log^{\beta_1} x$ and $g_2(x)=x^\alpha \log^{\beta_2} x,$ with $\alpha>0,$ $\alpha\notin \N$ and $\beta_1>\beta_2.$ Then $g_1, g_2\in\CT_{[\alpha]}$ with $g_2\prec g_1,$ and $\lim_{x\to \infty}(xg'_i(x)/g_i(x))=\alpha,$ $i=1,2.$)  

\medskip

(ii) If $g_2\prec g_1\in \mathcal{T},$ with $g_2\in \mathcal{T}_{i_0},$ then $g_2^{(k)}\prec g_1^{(k)}$ for all $k=0,\ldots,i_0+1.$ Indeed, assuming that $\lim_{x\to\infty}\frac{xg_i'(x)}{g_i(x)}=\alpha_i,$ $i=1,2,$ (i) implies $\alpha_2\leq \alpha_1$. The claim follows by
\[\frac{g_1^{(k)}(x)}{g_2^{(k)}(x)}= \left(\prod_{i=1}^k\frac{xg_1^{(i)}(x)}{g_1^{(i-1)}(x)}\cdot  \frac{g_2^{(i-1)}(x)}{xg_2^{(i)}(x)}\right)\cdot \frac{g_1(x)}{g_2(x)}\]  (note here the crucial fact that $k\leq i_0+1$ so $\lim_{x\to\infty}xg_i^{(k)}(x)/g_i^{(k-1)}(x)$ is never $0,$ hence we can freely take limits of ratios of such expressions).\footnote{ In order to show that $g_2^{(k)}\prec g_1^{(k)}$ for all $k,$ we can restrict to the subfamily of $\mathcal{T}$ of functions $g$ where $\lim_{x\to\infty}xg'(x)/g(x)$ is strictly between $i$ and $i+1$ for some $i\geq 0$ (see below, IV (ii), for some weaker assumption). When one deals with, say logarithmico-exponential, $\mathcal{LE}$, Hardy field functions though there are no such issues, as this class is closed under derivatives and analogous expressions always have limits, hence the use of L'Hospital's rule is not restricted. On the other hand, dealing with the class of tempered functions, even the one of Fej{\'e}r ones, one can potentially have various issues. For example, take the sublinear function $g_1(x)=x/\log x\in \mathcal{T}_0.$ Not only $g_1'$ is not a tempered function but also $xg_1''(x)/g_1'(x)$ converges to $0.$ There are even more exotic cases. Namely, take $g_2(x)=x^\alpha(4/\alpha+\sin\log x)^3,$ where $\alpha$ is a sufficiently small positive real number (this special function was firstly introduced and studied in \cite{DKS}). While $g_2$ is Fej{\'e}r, the quantity $xg_2'(x)/g_2(x)$ doesn't even have a limit. Cases like this last one need to (and will) be avoided in our framework.}

\medskip

(iii) If $g\in \CT$ then $g^{(j)}(x)/g^{(i)}(x)\to 0$ for all $j>i\in \omega$ (so, $g^{(j)}(x)\to 0$ for all $j>i_0,$ if $g\in \CT_{i_0}$). Indeed, this follows from the relation:
\[\frac{g^{(j)}(x)}{g^{(i)}(x)}=\left(\prod_{k=i}^{j-1}\frac{xg^{(k+1)}(x)}{g^{(k)}(x)}\right)\cdot\frac{1}{x^{j-i}}.\]

\medskip

For $g_1,$ $g_2$ defined on $\mathbb{R}^+,$ we write $g_1\sim g_2$ if $g_1(x)/g_2(x)$ converges to a non-zero constant as $x\to\infty.$

\medskip

(iv) If $g_1\in \mathcal{T}_{i_0}$ and $g_2\in \mathcal{R}$ with $g_2\sim g_1,$ then $g_2\in \mathcal{T}_{i_0}$ and $\lim_{x\to\infty}(xg'_2(x)/g_2(x))=\lim_{x\to\infty}(xg'_1(x)/g_1(x)).$ 

Indeed, since 
\[\frac{g'_2(x)}{g'_1(x)}=\frac{xg'_2(x)}{g_2(x)}\cdot \frac{g_1(x)}{xg'_1(x)}\cdot \frac{g_2(x)}{g_1(x)},\] we have that the limit $\lim_{x\to\infty}(g'_2(x)/g'_1(x))$ exists  (and so it has the same value as the one of $g_2(x)/g_1(x)$). The second claim now follows by the relation
\[ \frac{x g_2'(x)}{g_2(x)}=\frac{xg_1'(x)}{g_1(x)}\cdot \frac{g_2'(x)}{g_1'(x)} \cdot \frac{g_1(x)}{g_2(x)},\] so $g_2\in\mathcal{T}_{i_0}$ since 
\[g_2^{(i_0+1)}(x)=\frac{g_2^{(i_0+1)}(x)}{g_1^{(i_0+1)}(x)}\cdot g_1^{(i_0+1)}(x)=\left(\prod_{i=1}^{i_0}\frac{xg_2^{(i+1)}(x)}{g_2^{(i)}(x)}\cdot\frac{g_1^{(i)}(x)}{xg_1^{(i+1)}(x)}\right)\cdot\frac{g_2'(x)}{g_1'(x)}\cdot g_1^{(i_0+1)}(x).\]

\medskip

(v) It is easy for one to check that $g\in \CT_{i+1}\;\Leftrightarrow\;g'\in \CT_i,$ $i\geq 0.$\footnote{ This simple observation is crucial for the sequel, as it shows that the differences (i.e., derivatives) reduce the complexity of a function in $\CT.$ Note also that analogously to the previous property we have that $g\in \CT_{i}\;\Leftrightarrow\;xg\in \CT_{i+1},$ $i\in \omega$ (resp. $g\in \mathcal{T}_{i+1} \; \Leftrightarrow\; g/x\in \mathcal{T}_i,$ $i\in \omega$), using the fact that $\frac{d^k}{dx^k}(xg(x))=kg^{(k-1)}(x)+xg^{(k)}(x)$ for all $k\in \mathbb{N}.$}

\subsection*{IV. Linear combinations of functions of $\CT$}

(i) If $g_1\in \CT_{i_0},$ for some $i_0\in \omega,$ then any non-trivial linear combination (i.e., not all coefficients are equal to $0$) of the form 
\[g=\lambda_0 g_1+\ldots+\lambda_{i_0+1}g_1^{(i_0+1)},\] belongs to $\mathcal{T}_{i_0-k^\ast}$ for $k^\ast=\min\{i: \;\lambda_i\neq 0\}$ (where we set $\CT_{-1}$ to be the set of functions that converge to $0$ to which we artificially assign the degree of $-1$).

Using III (iii), we have that \[\frac{g(x)}{g_1^{(k^\ast)}(x)}=\sum_{j=k^\ast}^{i_0+1}\lambda_j\cdot\frac{g_1^{(j)}(x)}{g_1^{(k^\ast)}(x)}\to \lambda_{k^\ast}\in \mathbb{R},\] so $g\sim g_1^{(k^\ast)},$ hence the claim will follow from III (iv) and (v) if we show that $g\in \mathcal{R}.$ This follows by the relation 
\[\frac{xg^{(j+1)}(x)}{g^{(j)}(x)}=\frac{xg_1^{(k^\ast+j+1)}(x)}{g_1^{(k^\ast+j)}(x)}\cdot\frac{1+\sum_{i=k^\ast+1}^{i_0+1}\frac{\lambda_i}{\lambda_{k^\ast}}\cdot\frac{g_1^{(i+j+1)}(x)}{g_1^{(k^\ast+j+1)}(x)}}{1+\sum_{i=k^\ast+1}^{i_0+1}\frac{\lambda_i}{\lambda_{k^\ast}}\cdot\frac{g_1^{(i+j)}(x)}{g_1^{(k^\ast+j)}(x)}},\] and III (iii).

\medskip

(ii) If $g_j\in \CT_{i_j},$ $j=1,\ldots,\ell$ with $g_\ell\prec\ldots\prec g_1$ and $g_j^{(i_j+2)}(x)/g_1^{(i_j+2)}(x)\to 0,$ for all $1<j\leq \ell,$ we have that any linear combination of the form
\[g=\lambda_1 g_1+\ldots +\lambda_\ell g_\ell\] with $\lambda_1\neq 0,$ is a function in $\CT_{i_1}.$\footnote{ Hence provides a sufficient condition for Theorems~\ref{T:vNt} and ~\ref{T:main_general}.}

Indeed, since $g\sim g_1,$ by III (iv), it suffices to show that $g\in \mathcal{R}.$ This follows by the relation 
\[\frac{xg^{(j+1)}(x)}{g^{(j)}(x)}=\frac{xg_1^{(j+1)}(x)}{g_1^{(j)}(x)}\cdot\frac{1+\sum_{i=2}^{\ell}\frac{\lambda_i}{\lambda_{1}}\cdot\frac{g_i^{(j+1)}(x)}{g_1^{(j+1)}(x)}}{1+\sum_{i=2}^{\ell}\frac{\lambda_i}{\lambda_{1}}\cdot\frac{g_i^{(j)}(x)}{g_1^{(j)}(x)}},\] noting that $g_j^{(k)}\prec g_1^{(k)}$ for all $k\in \omega,$ $1<j\leq \ell,$
which holds from III (ii) for $k\leq i_j+1,$ the assumption for $k=i_j+2,$ and the relation 
\[\frac{g_{1}^{(k)}(x)}{g_j^{(k)}(x)}= \left(\prod_{i=i_j+3}^k\frac{xg_1^{(i)}(x)}{g_1^{(i-1)}(x)}\cdot  \frac{g_j^{(i-1)}(x)}{xg_j^{(i)}(x)}\right)\cdot \frac{g^{(i_j+2)}_1(x)}{g_j^{(i_j+2)}(x)}\] otherwise.

\medskip

(iii) Let $g_j\in \CT_{i_0},$ $1\leq j\leq \ell,$ with $g_\ell\prec\ldots\prec g_1.$ If every non-trivial linear combination of $g_i$'s is in $R,$ then any linear combination of the form
\[\lambda_1 g^{(i_0)}_1+\ldots +\lambda_\ell g^{(i_0)}_\ell\] with $\lambda_1\neq 0,$ is a function in $\CT_{0}.$

Indeed, letting $g=\lambda_1 g_1+\ldots +\lambda_\ell g_\ell,$ since $g_j^{(i_0)}\prec g_i^{(i_0)}$ for all $j>i$ by III (ii), we have that $g^{(i_0)}\sim g_1^{(i_0)}\in \CT_0.$ Using the assumption, we have that $g\in \mathcal{R},$ hence $g^{(i_0)}$ is also in $\mathcal{R},$ so the result now follows from III (iv). 

\subsection{The van der Corput lemma}\label{vdC,FCP}\footnote{ Also known as \emph{van der Corput trick}.} In this subsection, we present a crucial tool for our study, namely, the van der Corput lemma. More specifically, we will use the following version of it (for a proof, one can imitate \cite[Lemma~3.1]{Nid}):

\begin{lemma}\label{vdClemma}
Let $(u_n)_n$ be a bounded sequence in a Hilbert space $(\mathcal{H},\norm{\cdot}).$ We have that
\[\limsup_{N\to\infty}\norm{\frac{1}{N}\sum_{n=1}^N u_n}^2\leq 4\limsup_{H\to\infty}\frac{1}{H}\sum_{h=1}^H \limsup_{N\to\infty}\left|\frac{1}{N}\sum_{n=1}^N \langle u_n,u_{n+h}\rangle \right|.\]
\end{lemma}

Its iterated use will allow us to eventually reduce the complexity of our system sufficiently enough.
What we do is that we shift our functions and calculate the inner product of the new shifted expression with the initial one. Letting, 
for $h\in\mathbb{R}$ and $g\in \mathcal{T},$  \[(S_h g)(x)=g(x+h),\footnote{ Even though some results are stated generally for $h\in \mathbb{R},$ we will actually use them only for $h\in \mathbb{N}$.}\]
 we naturally define, following the definition from \cite{F3}, the van der Corput operation:

\medskip

Let $(\CA_1,\ldots,\CA_\ell):=((g_{1,1},\ldots,g_{\ell,1}),\ldots,$ $(g_{1,m},\ldots,g_{\ell,m}))$ be an ordered family of $\ell$-tuples of functions,\footnote{ Note that each $\mathcal{A}_j$ records the iterates of the function $f_j$ in \eqref{Prop: nice}, $1\leq j\leq \ell$.} $(\tilde{g}_1,\ldots,\tilde{g}_\ell)\in (\CA_1,\ldots,\CA_\ell),$ i.e., $(\tilde{g}_1,\ldots,\tilde{g}_\ell)=(g_{1,j},\ldots,g_{\ell,j})$ for some $1\leq j\leq m,$  and $h\in \mathbb{N}.$ The \emph{van der Corput operation} (\emph{vdC-operation}) acting on  $(\CA_1,\ldots,\CA_\ell),$ gives the family 
\[ (\tilde{g}_1,\ldots,\tilde{g}_\ell,h)\text{-}\text{vdC}(\CA_1,\ldots,\CA_\ell)
 \] of the ordered $\ell$-tuples:
 \[ \{(S_h g_{1,i}-\tilde{g}_1,\ldots,S_h g_{\ell,i}-\tilde{g}_\ell):\;i\in I\}\cup\{(g_{1,j}-\tilde{g}_1,\ldots,g_{\ell,j}-\tilde{g}_\ell):\; j\in J\}\] where $I, J\subseteq \{1,\ldots,m\}$ from which we have discarded the $\ell$-tuples of bounded functions.\footnote{ This removal will be justified later by the use of the Cauchy-Schwarz inequality.}

Later (see Section~\ref{slF}), after defining what a nice family of tempered functions is,\footnote{ Yet again following the corresponding definition for Hardy field functions from \cite{F3}.} we will show that this notion is preserved under a special, in the sense that it reduces the complexity, vdC-operation which will provide the required inductive scheme.

\medskip

One may wonder, since, after applying the vdC-operation, we are getting (up to) double the number of $\ell$-tuples of iterates (the initial ones together with their shifts) how one succeeds in reducing the ``complexity'' of the expressions of interest.\footnote{ This ``complexity'' is what we later define (in Section~\ref{S:nice}) as weight of the family.} This is achieved by the fact that our transformations are measure preserving, hence we can always subtract an $\ell$-tuple of iterates, discarding also the bounded ones.
 For differences of the same function one can easily show (see Lemma~\ref{L:1} below) that if $g\in \CT_i,$ then for any non-zero $h\in \mathbb{R}$ we have that $S_h g - g \sim g'\sim g/x\in \mathcal{T}_{i-1},$ hence the new iterate is of lower complexity. More generally, by carefully picking an $\ell$-tuple, we will show (imitating the proof of \cite[Lemma~5.5]{F3}) in Lemma~\ref{L:vdC red} that we get a new family with the required property. We also show that this can be done by simultaneously preserving the niceness property.

\subsection{Background material} In this subsection we list some background material that we use throughout the paper.

\subsubsection{The seminorms $\nnorm{\cdot}_k$}\label{Ss:semi}

We follow \cite{HK99} and \cite{CFH} for the inductive definition of the seminorms $\nnorm{\cdot}_k$ that we will use to control our averages. More specifically, the definition that we use here follows from \cite{HK99} (in the ergodic case), \cite{CFH} (in the general case) and the use of von Neumann's ergodic theorem.

\medskip

Let $(X,\mathcal{B},\mu,T)$ be a system and $f\in L^\infty(\mu).$  We define inductively the seminorms $\nnorm{f}_{k,\mu,T}$ (or just $\nnorm{f}_k$ if there is no confusion) as follows: 
$$ \nnorm{f}_{1,\mu,T}:= \norm{\E(f|\mathcal{I}(T))}_{2},$$
where $\mathcal{I}(T)$ is the $\sigma$-algebra of $T$-invariant sets and $\E(f|\mathcal{I}(T))$ the conditional expectation of $f$ with respect to $\mathcal{I}(T),$ satisfying $\int \E(f|\mathcal{I}(T))\;d\mu=\int f\;d\mu$ and $ T\E(f|\mathcal{I}(T))=\E(Tf|\mathcal{I}(T)).$

\medskip

For $k\geq 1,$ we let
\[\nnorm{f}^{2^{k+1}}_{k+1,\mu,T}:=\lim_{N-M\to\infty}\frac{1}{N-M}\sum_{n=M}^{N-1}\nnorm{\bar{f}\cdot T^n f}^{2^k}_{k,\mu,T}.\]
All the aforementioned limits exist and define seminorms (see \cite{HK99}). By using von Neumann's ergodic theorem, we get $\nnorm{f}^2_{1,\mu,T}=\lim_{N-M\to\infty}\frac{1}{N-M}\sum_{n=M}^{N-1}\int \bar{f}\cdot T^n f\;d\mu$ and, more generally, for every $k\geq 1$ we have that
\begin{equation}\label{E:norms}
\nnorm{f}^{2^k}_{k,\mu,T}=\lim_{N-M\to\infty}\frac{1}{N-M}\sum_{n_1=M}^{N-1} \cdots \lim_{N-M\to\infty}\frac{1}{N-M}\sum_{n_k=M}^{N-1}\int \prod_{\vec{\epsilon}\in \{0,1\}^k}\mathcal{C}^{|\vec{\epsilon}|}T^{\vec{\epsilon}\cdot\vec{n}}f\;d\mu,
\end{equation}
 where $\vec{\epsilon}=(\epsilon_1,\ldots,\epsilon_k),$ $\vec{n}=(n_1,\ldots,n_k),$ $|\vec{\epsilon}|=\epsilon_1+\ldots+\epsilon_k,$ $\vec{\epsilon}\cdot\vec{n}=\epsilon_1 n_1+\ldots+\epsilon_k n_k$ and for $z\in \C,$ $k\in\omega$ we let $ \mathcal{C}^k z=\left\{ \begin{array}{ll} z& \quad \;\text{if}\;k\;\text{is even} \\ \bar{z}&  \quad \;\text{if}\;k\;\text{is odd}\end{array} \right.$.

\medskip

Also, we remark that $\nnorm{f\otimes\bar{f}}_{k,\mu\times\mu,T\times T}\leq \nnorm{f}^2_{k+1,\mu,T}$ and $\nnorm{f}_{k,\mu,T}=\nnorm{f}_{k,\mu,T^{-1}}$ for all $k\in \N,$ which follow from \eqref{E:norms} and the ergodic theorem, 
 and, finally, $\nnorm{f}_{k,\mu,T}\leq \nnorm{f}_{k+1,\mu,T}$ for all $k\in\N$ (using \cite[Lemma~3.9]{HK99}).



\subsubsection{Nilmanifolds and nilsequences}

Let $ G $ be a $ k $-step nilpotent Lie group, meaning $ G_{k+1}=\{e\} $ for some $ k\in \N $, where $ G_k=[G,G_{k-1}] $ denotes the $ k $-th commutator subgroup, and $ \Gamma $ a discrete cocompact subgroup of $ G $. The  compact homogeneous space
$ X=G/\Gamma $ is called \emph{$k$-step nilmanifold} (or just \emph{nilmanifold}). A \emph{$k$-step nilsequence} is a sequence of the form  $(F(g^n x))_n,$ where $F$
is a continuous function on a $k$-step nilmanifold X. The group $ G $ acts on $ G/\Gamma $ by left translation where the translation  by an element $ b\in G $ is given by $ T_b(g\Gamma)=(bg)\Gamma $. We denote by $ m_X $ the normalized \emph{Haar measure} on $ X,$ i.e., the unique probability measure that is invariant under the action of $ G $ by left translations. 

\subsubsection{Equidistribution on nilmanifolds}\label{Ss:ud} For a connected and simply connected Lie group $G,$ let $ \exp :\mathfrak{g}\to G $ be the exponential map, where $ \mathfrak{g} $ is the Lie algebra of $ G $. For $ b \in G $ and $ s\in \R $ we define the element $ b^s $ of $ G $ as follows: If $ X\in \mathfrak{g} $ is such that $ \exp (X)=b $, then $ b^s=\exp(sX) $ (this is well defined since under the aforementioned assumptions $ \exp$ is a bijection). If $ (a(n))_{n} $ is a sequence of real numbers and $ X=G/\Gamma $ is a nilmanifold with $ G $ connected and simply connected, we say that the sequence
$ (b^{a(n)}x)_{n} $ is \emph{equidistributed} in $ X $, if for every $ F\in C(X) $ we have
\begin{equation}\label{E:equi}
    \lim_{N\to \infty} \frac{1}{N}\sum_{n=1}^N F(b^{a(n)}x)=\int F\; d m_X.\footnote{ If the sequence $ (a(n))_{n} $ takes only integer values, we are not obliged to assume that $ G $ is connected and simply connected.}
\end{equation}  


\noindent A nilrotation $ b\in G $ is \emph{ergodic}, or \emph{acts ergodically} on $ X $, if the sequence $ (b^n\Gamma)_{n} $ is dense in $X.$ If $ b\in G $ is ergodic, then for every
$ x\in X $ the sequence $ (b^nx)_{n} $ is equidistributed in $ X $.



The orbit closure
$\overline{(b^n\Gamma)}_{n} $ of $ b \in G$ has the structure of a nilmanifold; furthermore, the sequence $ (b^n\Gamma)_{n} $ is equidistributed in $ \overline{(b^n\Gamma)}_{n}$. If $ G $ is connected and simply connected and $ b\in G $, then
$ \overline{(b^s\Gamma)}_{s\in \R} $ is a nilmanifold with the nilflow $ (b^s\Gamma)_{s\in \R} $ being equidistributed in $ \overline{(b^s\Gamma)}_{s\in \R} $.
For the special case of $\T^\ell,$ by \cite{Weyl}, an equivalent to \eqref{E:equi} condition for $ (a(n))_{n}\subseteq \R^\ell $ to be equidistributed in $ \T^\ell $ (or equidistributed ($\text{mod}$ 1)) is to satisfy \emph{Weyl's criterion}, i.e., 
\[ \lim_{N\to \infty}\frac{1}{N}\sum_{n=1}^Ne^{2\pi i a(n)\cdot \vec{h}}=0, \]
for every non-zero $ \vec{h}\in \Z^\ell $, where $ a(n)\cdot \vec{h} $ denotes the inner product of
 $ a(n) $ with $\vec{h}$.






\section{A von Neumann-type result for tempered functions}\label{vN}

This short section is dedicated to the proof of Theorem~\ref{T:vNt}. We first recall a Hilbert space splitting theorem and a version of the classical Bochner-Herglotz theorem which we'll use for the space $L^2.$

\begin{theorem}[\cite{B1}]\label{T:split}
For $\ell\in \mathbb{N}$ let $U_1,\ldots,U_\ell$ be commuting unitary operators on a Hilbert space $(\mathcal{H},\norm{\cdot}).$ If 
\[\mathcal{H}_{\text{inv}}:=\{f\in \mathcal{H}:\;U_i f=f\;\;\text{for all}\;\;1\leq i\leq \ell\},\] and 
\[\mathcal{H}_{\text{erg}}:=\left\{f\in \mathcal{H}:\;\lim_{N_1,\ldots,N_\ell\to\infty}\norm{\frac{1}{N_1\cdots N_\ell}\sum_{n_1=1}^{N_1}\cdots\sum_{n_\ell=1}^{N_\ell}U_1^{n_1}\cdots U_\ell^{n_\ell}f}=0\right\},\] then
\[\mathcal{H}=\mathcal{H}_{\text{inv}}\oplus \mathcal{H}_{\text{erg}}.\]
\end{theorem}

\begin{theorem}\label{T:BH}
For $\ell\in \mathbb{N}$ let $U_1,\ldots,U_\ell$ be commuting unitary operators on a Hilbert space $\mathcal{H}$ and $f\in \mathcal{H}.$ There exists a measure $\nu_f$ on $\mathbb{T}^\ell$ such that
\[\langle U_1^{n_1}\cdots U_\ell^{n_\ell}f,f\rangle=\int_{\mathbb{T}^\ell}e^{2\pi i(n_1\gamma_1+\ldots+n_\ell \gamma_\ell)}\;d\nu_f(\gamma_1,\ldots,\gamma_\ell),\] 
for all $(n_1,\ldots,n_\ell)\in \mathbb{Z}^\ell.$
\end{theorem}

We also prove the following Weyl-type result, which reveals the equidistribution properties of our functions:

\begin{proposition}\label{P:vn}
Under the assumptions of Theorem~\ref{T:vNt}, we have that
\begin{equation}\label{udg}
    \lim_{N\to\infty}\frac{1}{N}\sum_{n=1}^N e^{2\pi i ([g_1(n)]\gamma_1+\ldots+[g_\ell(n)]\gamma_\ell)}=0,
\end{equation}
for all $\gamma_1,\ldots,\gamma_\ell \in \mathbb{R}\setminus\mathbb{Z}.$
\end{proposition}

\begin{proof} We actually present here the $\ell=2$ case for convenience as it contains all the details for the general $\ell\in\mathbb{N}$ statement. To do so we split the proof into three cases.

\medskip

{\bf Case 1.} $\gamma_1,$ $\gamma_2\in \mathbb{R}\setminus\mathbb{Q}.$ 

\medskip

The proof of this step follows \cite[Lemma~5.12]{BK}. To show \eqref{udg} it suffices to show that $([g_1(n)]\gamma_1, [g_2(n)]\gamma_2)$ is equidistributed in $\mathbb{T}^2$ (which is actually a characterization that the numbers $\gamma_1,$ $\gamma_2\in \mathbb{R}\setminus\mathbb{Q}$). So, it suffices to show that the sequence \[\left(\gamma_1 g_1(n), g_1(n), \gamma_2 g_2(n), g_2(n)\right)_n\] is equidistributed in $\mathbb{T}^4,$ which is true if and only if \[\left((a\gamma_1+b) g_1(n)+(c\gamma_2+d) g_2(n)\right)_n\] is equidistributed in $\mathbb{T}$ for all $(a,b,c,d)\in \mathbb{Z}^4\setminus\{(0,0,0,0)\}.$ Since $\gamma_1,$ $\gamma_2\in \mathbb{R}\setminus\mathbb{Q},$ using the assumption on the $g_i$'s, we have that $(a\gamma_1+b) g_1+(c\gamma_2+d) g_2$ is a tempered function, so the result follows by I (iv).


\medskip

{\bf Case 2.} $\gamma_1,$ $\gamma_2\in \mathbb{Q}\setminus\mathbb{Z}.$

\medskip

If $\gamma_i=\frac{p_i}{q_i},$ we set $m=q_1q_2.$ Using some algebra, \eqref{udg} follows in this case if we show that  \[\lim_{N\to \infty} \frac{1}{N}\sum_{n=1}^N e^{2\pi i\frac{1}{m}(h_1[g_1(n)]+h_2[g_2(n)])}=0,\] for all $m\geq 2,$ $1\leq h_1, h_2\leq m-1.$

\medskip 


By the assumption on the $g_i$'s, we have that $(g_1(n)/m,g_2(n)/m)$ is equidistributed in $\mathbb{T}^2$ for all $m\geq 2.$ 
So, since $[x]=[\frac{x}{m}]m+j$ if $\frac{j}{m}\leq \left\{\frac{x}{m}\right\}\leq \frac{j+1}{m},$ $0\leq j\leq m-1,$ where $\{\cdot\}$ denotes the fractional part function, setting $E_{j_1,j_2}=[\frac{j+1}{m},\frac{j_1+1}{m})\times [\frac{j_2}{m},\frac{j_2+1}{m}),$ $0\leq j_1, j_2\leq m-1,$ we have:
	\begin{equation*}
	\begin{split}
	&\quad \lim_{N\to \infty} \frac{1}{N}\sum_{n=1}^N e^{2\pi i\frac{1}{m}(h_1[g_1(n)]+h_2[g_2(n)])}
	\\&= \lim_{N\to \infty} \frac{1}{N}\sum_{n=1}^N\sum_{j_1,j_2=0}^{m-1} e^{2\pi i\frac{1}{m}(h_1([\frac{g_1(n)}{m}]m+j_1)+h_2([\frac{g_1(n)}{m}]m+j_2))}{\bf{1}}_{E_{j_1,j_2}}(\{g_1(n)/m\},\{g_2(n)/m\})
		\\&= \lim_{N\to \infty}\frac{1}{N}\sum_{n=1}^N\sum_{j_1,j_2=0}^{m-1} e^{2\pi i\frac{1}{m}(h_1 j_1+h_2 j_2)}{\bf{1}}_{E_{j_1,j_2}}(\{g_1(n)/m\},\{g_2(n)/m\})
		\\&=\sum_{j_1,j_2=0}^{m-1} e^{2\pi i\frac{1}{m}(h_1 j_1+h_2 j_2)}\int_0^1\int_0^1 {\bf{1}}_{E_{j_1,j_2}}(x,y)\;dx dy=\frac{1}{m^2}\sum_{j_1,j_2=0}^{m-1} e^{2\pi i\frac{1}{m}(h_1 j_1+h_2 j_2)}=0.
	\end{split}
    \end{equation*}


{\bf Case 3.} $\gamma_1\in \mathbb{R}\setminus\mathbb{Q}$ and $\gamma_2\in \mathbb{Q}\setminus\mathbb{Z}.$\footnote{ The case where $\gamma_2\in \mathbb{R}\setminus\mathbb{Q}$ and $\gamma_1\in \mathbb{Q}\setminus\mathbb{Z}$ is analogous since we only care about the growth rates of the functions which are different.}

\medskip

For all $m\geq 2$ we have that $(g_1(n)\gamma_1,g_1(n),g_2(n)/m)_n$ is equidistributed in $\mathbb{T}^3$ since $(a\gamma_1+b)g_1+\frac{c}{m}g_2$ is a tempered function for all $(a,b,c)\in\mathbb{Z}^3\setminus \{(0,0,0)\}.$ It follows that $(g_1(n)\gamma_1,g_1(n),[g_2(n)])_n$ is equidistributed in $\mathbb{T}^2\times \mathbb{Z},$ therefore \[\lim_{N\to\infty}\frac{1}{N}\sum_{n=1}^N e^{2\pi i(\gamma_1[g_1(n)]+\frac{h}{m}[g_2(n)])}=0,\] for all $m\geq 2,$ $1\leq h\leq m-1,$ hence \eqref{udg} follows.
\end{proof}

We are now ready to prove Theorem~\ref{T:vNt}.

\begin{proof}[Proof of Theorem~\ref{T:vNt}]
Using Theorem~\ref{T:split}, we have that $L^2(\mu)=L^2(\mu)_{\text{inv}}\oplus L^2(\mu)_{\text{erg}}.$ For $f\in L^2(\mu)_{\text{inv}}$ we have that $T_1^{[g_1(n)]}\cdots T_\ell^{[g_\ell(n)]}f=f,$ so, it suffices to show that for $f\in L^2(\mu)_{\text{erg}}$ we have 
\[\lim_{N\to\infty}\norm{\frac{1}{N}\sum_{n=1}^N T_1^{[g_1(n)]}\cdots T_\ell^{[g_\ell(n)]}f}_2=0.\] This follows though from Theorem~\ref{T:BH} and Proposition~\ref{P:vn}. Indeed,
\begin{eqnarray*}
\norm{\frac{1}{N}\sum_{n=1}^N \left(\prod_{i=1}^\ell T_i^{[g_i(n)]}\right)f}^2_2 & = &  \frac{1}{N^2}\sum_{n,m=1}^N \langle \left(\prod_{i=1}^\ell T_i^{[g_i(n)]-[g_i(m)]}\right)f,f\rangle \\
& = & \frac{1}{N^2}\sum_{n,m=1}^N \int e^{2\pi i\left(\sum_{i=1}^\ell ([g_i(n)]-[g_i(m)])\cdot \gamma_i\right)}\;d\nu_f(\gamma_1,\ldots,\gamma_\ell) \\
& = & \int \left| \frac{1}{N}\sum_{n=1}^N e^{2\pi i\left(\sum_{i=1}^\ell [g_i(n)]\cdot \gamma_i\right)}\right|^2 \;d\nu_f(\gamma_1,\ldots,\gamma_\ell)
\end{eqnarray*}
which goes to $0$ as $N\to \infty$ since $f\in L^2(\mu)_{\text{erg}},$ hence $\nu_f(\{(0,\ldots,0)\})=0.$
\end{proof}

\section{The sub-linear case, Fej{\'e}r functions}\label{slF}

In this short section we treat the sub-linear case, i.e., when all the functions are Fej\'{e}r, separately. The main reason for doing this is that in this case, as in \cite{F2} for sub-linear Hardy field functions, we can show convergence to the expected limit, without using any commutativity assumptions on the $T_i$'s. We will actually prove this (Theorem~\ref{T:sub-linear}) for functions from $\mathcal{F},$ following arguments from \cite{F2}, \cite{F3} and \cite{BK}.\footnote{ In \cite{DKS} we dealt with the (a.e.) pointwise convergence of averages, addressing the sub-linear case for a large family of functions which implies the corresponding to Theorem~\ref{T:sub-linear} result for functions from $\mathcal{T}_0$.} This result will also be used in the proof of the base case of Proposition~\ref{P:1}.  

\begin{lemma}\label{L:sub-linear}
For $\ell\in \mathbb{N}$ let $g_\ell\prec \ldots\prec g_1$ functions from $\CF$ with $g'_i/g'_j$ eventually monotone for all $j\leq i,$ and $(A_i(n))_n$ sequences of functions in $L^\infty(\mu)$ with uniformly bounded norm such that the limits $\tilde{A}_i=\lim_{N-M\to\infty}\frac{1}{N-M}\sum_{n=M}^{N-1} A_i(n),$ $1\leq i\leq \ell,$ exist in $L^2(\mu).$ Then \[\lim_{N\to\infty}\norm{\frac{1}{N}\sum_{n=1}^N A_1([g_1(n)])\cdots A_\ell([g_\ell(n)])-\prod_{i=1}^\ell \tilde{A}_i}_2=0.\]
\end{lemma}

\begin{proof}
We use induction on $\ell.$ Assuming that $\tilde{A}_1=0$ we will show that the required limit is equal to $0.$
For $\ell=1$ and $g_1$ is positive, we define $\phi(n):=|\{m\in \mathbb{N}:\;[g_1(m)]=n\}|$ and $\Phi(n):=\sum_{k=0}^n \phi(k).$ By \cite[Theorem~3.5]{BK}, we have that \[\lim_{N\to\infty}\norm{\frac{1}{N}\sum_{n=1}^N A_1([g_1(n)])}_2=\lim_{N\to\infty}\norm{\frac{1}{\Phi(N)}\sum_{n=1}^N \phi(n)A_1(n)}_2.\] We have to show that this last limit is equal to $0.$ This follows from the hypothesis by using \cite[Theorem~3.6]{BK}, since, from \cite[Lemma~2.5]{BK}, we have that $\lim_{n\to\infty}\frac{\phi(n)}{\Phi(n)}=0$ and from \cite[Lemma~2.4]{BK} that $\lim_{n\to\infty}\phi(n)=\infty.$

The case where $g_1$ is negative, follows by the fact that $[g_1(n)]=-[-g_1(n)]-1$ in a set of density $1.$ Indeed, $[g_1(n)]=-[-g_1(n)]$ only happens when $g_1(n)$ is an integer, i.e., at most $g_1(N)$ times up to time $N$. The claim now follows by the sub-linearity of $g_1.$\footnote{ This base case can be viewed as a ``change of variable'' method
with the crucial remark that the uniform Ces\'{a}ro average is been replaced with a standard one (see also the proof of Proposition~\ref{P:1}).}

Assuming that the result holds for $\ell-1$ terms, we will show that $$\lim_{N\to\infty}\norm{\frac{1}{N}\sum_{n=1}^N A_1([g_1(n)])\cdots A_\ell([g_\ell(n)])}_2=0.$$ For $2\leq i\leq \ell$ we let $\tilde{g}_i(x)=g_i(g^{-1}_1(x))$ which belong, by \cite[Lemma~2.7]{BK}, in $\CF.$ By \cite[Lemma~2.12]{BK}, we have that $$[g_i(n)]=[\tilde{g}_i([g_1(n)])]$$ for a set of $n$'s of density $1.$
By the $\ell=1$ case it suffices to show that \[\frac{1}{|I_N|}\sum_{n\in I_N}  A_1(n)\cdot A_2([\tilde{g}_{2}(n)])\cdots A_{\ell}([\tilde{g}_{\ell}(n)])\] converges to $0$ in $L^2(\mu)$ as $N\to\infty,$ where $(I_N)_N$ is a sequence of intervals of integers with lengths increasing to infinity.

Since $\tilde{g}_i(x+1)-\tilde{g}_i(x)$ converge to $0$ and have eventually constant sign, as in the proof of the base case of \cite[Proposition~4.2]{F3}, each interval $I_N,$ $N\in \mathbb{N},$ can be decomposed (except a finite set of fixed cardinality) into subintervals with lengths tending (as $N\to\infty$) to infinity, in such a way that the sequences $([\tilde{g}_i(n)])_n,$ $2\leq i\leq \ell,$ are constant on each of the subintervals.   So, without loss of generality, we may assume that  there exist sequences of integers $(c_{i,N})_N,$ $2\leq i\leq \ell,$ such that $[\tilde{g}_i(n)]=c_{i,N}$ for all $n\in I_N.$ Using the fact that the quantity $ A_2(c_{2,N})\cdots A_{\ell}(c_{\ell,N})$ has uniformly bounded $L^\infty(\mu)$-norm and the hypothesis on the uniform convergence of the average of $A_1(n),$ we get the result. 
\end{proof}

\begin{remark*} By the relation 
\[\frac{g'_{i}}{g'_j}=\prod_{k=j}^{i-1} \left(\frac{g'_{k+1}}{g'_k}\right),\] and working inductively, using also the fact that a function which converges monotonically to $0$ has the opposite sign from its derivative, we have that the assumption ``$g'_{i+1}/g'_i$ is eventually monotone,'' implies that ``$g'_i/g'_j$ is eventually monotone for any $j\leq i.$''
\end{remark*}

Immediate implication of Lemma~\ref{L:sub-linear}, together with the remark after it, is the following.

\begin{proof}[Proof of Theorem~\ref{T:sub-linear}]
For any $1\leq i\leq \ell$ we let $A_i(n)=T_i^n f_i.$ Using von Neumann's uniform mean ergodic theorem, we have $$\lim_{N-M\to\infty}\norm{\frac{1}{N-M}\sum_{n=M}^{N-1} A_i(n)-\mathbb{E}(f_i|\mathcal{I}(T_i))}_2=0,$$ for all $1\leq i\leq \ell,$ and the result follows by Lemma~\ref{L:sub-linear}.
\end{proof}

\begin{remark*}
As it was mentioned in Section~\ref{results}, the conclusion of Theorem~\ref{T:sub-linear} also holds for functions $g_\ell\prec \ldots\prec g_1\in \mathcal{T}_0.$

Indeed, this follows by the fact that, using III (ii), the corresponding pointwise result \cite[Theorem~1.1]{DKS} holds for functions $g_\ell\prec \ldots\prec g_1\in \mathcal{T}_0$  (via \cite[Proposition~2.2]{DKS} and \cite[Theorem~4.1]{DKS}--see comments after \cite[Corollary~4.2]{DKS}).  
\end{remark*}


\section{Nice families of tempered functions \\ and their invariance under the van der Corput operation}\label{S:nice}

As in the polynomial (\cite{CFH}) and Hardy field functions case (\cite{F3}), in order to show convergence under the weakly mixing assumption for multiple (commuting) $T_i$'s, we do so in a more general setting, namely, for a ``nice'' family of functions (see definition below).



\medskip

We start by reminding the reader that we denote the shift of a function $g$ by $h\in \mathbb{R}$ with $S_h g,$ and by recalling a result from \cite{BK} which will be used many times in what follows:

\begin{lemma}[\mbox{\cite[Lemma~2.3]{BK}}]\label{L:BK}
Let $g\in C^\infty(\mathbb{R}^+)$ such that $\lim_{x\to\infty}\frac{xg'(x)}{g(x)}\in\mathbb{R}.$ Then, for every $h\in \mathbb{R},$ $\lim_{x\to \infty}\frac{S_h g(x)}{g(x)}=1.$
\end{lemma}



\begin{definition*}
Let $g\in \CT$ and $\CC(g)$ be the family of functions which contains all integer combinations of shifts of $g,$ i.e.,
\[\CC(g)=\Big\{\sum_{i=1}^\ell k_i S_{h_i}g:\;k_i\in \Z,\;h_i\in \omega,\;\ell\in\N \Big\}.\]
\end{definition*}

A nice family of functions consists of linear combinations of functions from $\mathcal{T}$. 

\begin{definition*}\label{D:1}
Let $g_1,\ldots,g_\ell\in \CT,$ $g_{i,j}\in\CC(g_i)$ for $1\leq i\leq \ell,$ $1\leq j\leq m$ and $\CA_i:=(g_{i,1},\ldots,g_{i,m}),$ $1\leq i\leq \ell.$ We call the ordered family $(\CA_1,\ldots,\CA_\ell):=((g_{1,1},\ldots,g_{\ell,1}),\ldots,$ $(g_{1,m},\ldots,g_{\ell,m}))$\footnote{ 
Note that $\CA_i$ records the iterates of $T_i,$ while  $(\CA_1,\ldots,\CA_\ell)$ records the iterates of the products in the order that they appear in the expression of Proposition~\ref{P:1}.} of $\ell$-tuples of functions \emph{nice}  if:
\begin{enumerate}
\item $1\prec g_{1,1}-g_{1,j}$ and $g_{1,j}\ll g_{1,1}$ for $2\leq j\leq m;$

\item  $g_{i,j}\prec g_{1,1}$ for $2\leq i\leq \ell,$ $1\leq j\leq m;$ and

\item $g_{i,1}-g_{i,j}\prec g_{1,1}-g_{1,j}$ for $2\leq i\leq \ell,$ $2\leq j\leq m.$
\end{enumerate}

\end{definition*}

While running the PET induction, using vdC-operations, we deal with integer combinations of shifts of functions, so, the first step is to understand how the iterates behave and how their ``complexity'' changes through these operations. We first show some helpful lemmas and then, in Lemma~\ref{L:vdC red}, that the niceness notion is preserved under the vdC-operations and that this can be done in a way that the new nice family which is obtained has strictly smaller complexity than the previous one.

\medskip

It is easy to see that if $g\in \CT_i,$ $i\geq -1,$ then, for every $h\neq 0,$ we also have that $S_h g\in \CT_i.$ (Indeed, by Lemma~\ref{L:BK} we have that the limit, as $x\to\infty$, of
\[\frac{x (S_h g)^{(j+1)}(x)}{(S_h g)^{(j)}(x)}=\frac{x g^{(j+1)}(x)}{g^{(j)}(x)}\cdot\frac{g^{(j+1)}(x+h)}{g^{(j+1)}(x)}\cdot\frac{g^{(j)}(x)}{g^{(j)}(x+h)}\] exists for all $j\geq 0,$ hence $S_h g\in \mathcal{R}.$ The claim now follows by III (iv) as $S_h g\sim g$ again by Lemma~\ref{L:BK}.) The following lemma informs us about the ``order'' of the difference $S_h g - g.$

\begin{lemma}\label{L:1}
Let $g\in \CT.$ Then for any non-zero $h\in \mathbb{R}$ we have that $S_h g - g \sim g'\sim g/x.$
\end{lemma}

\begin{proof}
By the definition of $\CT,$ we immediately get that $g'\sim g/x.$ 

Suppose without loss of generality that $h>0.$ Using the mean value theorem we get \[ g(x+h)-g(x)=hg'(\xi_x) \] for some $\xi_x\in (x,x+h).$ Because of monotonicity, $g'(\xi_x)/g'(x)$ is squeezed between $1=g'(x)/g'(x)$ and $g'(x+h)/g'(x).$ Using Lemma~\ref{L:BK} we have that $g(x+h)/g(x)\to 1,$ so \[\frac{g'(x+h)}{g'(x)}=\frac{(x+h)g'(x+h)}{g(x+h)}\cdot \frac{g(x+h)}{g(x)}\cdot \frac{x}{x+h}\cdot \frac{g(x)}{xg'(x)}\to \alpha\cdot 1\cdot 1\cdot \frac{1}{\alpha}=1,\]  from where it follows that $S_h g-g\sim g'.$ 
\end{proof}

\begin{lemma}\label{L:-1}
Let $g\in \CT_i$ and $g_1\in \CC(g).$ Then either $g_1(x)\to 0$ as $x\to\infty,$ or there exists $i_0\in \{0,\ldots,i\}$ such that $g_1\sim g^{(i_0)}.$ In both cases, we have that $g_1\in \CT_{i-i_0}$ for some $0\leq i_0\leq i+1.$\footnote{ Hence, for each function $g_1\in  \CC(g),$ $g\in \CT,$ we have that $\deg g_1$ is well defined.}
\end{lemma}

\medskip

\begin{proof}
Using Taylor's theorem, each $S_h g$ has the form
\[S_h g(x)=\sum_{j=0}^{i}\frac{g^{(j)}(x)}{j!}h^j+\frac{g^{(i+1)}(\xi_{x,h})}{(i+1)!}h^{i+1},\] for some $\xi_{x,h}$ between $x$ and $x+h.$ Noting by the proof of the previous lemma that $g^{(i+1)}(\xi_{x,h})\sim g^{(i+1)}(x),$ writing $g_1$ in the form \begin{equation}\label{E:g-form}
    g_1(x)=\sum_{j=i_0}^i c_j g^{(j)}(x)+e(x),
\end{equation} 
 where $c_j\in \mathbb{R}$ and $e\sim g^{(i+1)},$ imitating the argument in IV (i), we have the first conclusion if every $c_j$ is $0,$ while the second one follows for the smallest $0\leq i_0\leq i$ with $c_{i_0}\neq 0.$ 
\end{proof}




 The form \eqref{E:g-form} will be used in the sequel and we will refer to it as the \emph{$g$-form} of $g_1.$
 
 \medskip
 
\noindent We now alter the definition from \cite{F3} of equivalent functions in order to fit our setting.
 
\begin{definition*} We say that $g_1, g_2\in \mathcal{C}(g),$ with $g_1, g_2\in \CT_{i_0}$ are \emph{equivalent}, and we write $g_1\cong g_2,$ if $g_1-g_2\in \CT_{j_0}$ for some $-1\leq j_0<i_0.$\footnote{ Reflecting once more the restrictions that one has when working with the class of functions $\CT,$ in \cite{F3} the notion of equivalence is defined in general for functions of polynomial growth. It is sufficient though for one to define it for functions $g_1, g_2\in \mathcal{C}(g)$ to have $g_1-g_2\in \CT$ (see remark after the definition).}
\end{definition*} 

\begin{remark*}
One here has to be careful with the condition $g_1, g_2\in \mathcal{C}(g)$ as the difference of two functions $g_1, g_2\in \CT_{i_0},$ $g_1-g_2,$ may not even be in $\CT.$   

Indeed, let for example $g_1(x)=x^{1/2}+\log^{1+\varepsilon}x$ and $g_2(x)=x^{1/2}.$ Then, even though $g_1, g_2\in \CT_0$ we have that $g_1-g_2\notin \CT.$

\medskip

On the other hand, let $g\in \CT_{i_0}$ and $g_1, g_2\in \mathcal{C}(g)$ with $1\prec g_1, g_2$ and $g$-forms
\[g_1(x)=\sum_{j=i_1}^{i_0} c_j g^{(j)}(x)+e_1(x),\;\;\text{and}\;\;g_2(x)=\sum_{j=i_2}^{i_0} c'_j g^{(j)}(x)+e_2(x).\]
We have that $g_1\ncong g_2$ iff either $i_1\neq i_2$ or $i_1=i_2$ and $c_{i_1}\neq c'_{i_1}.$ So, $g_1\cong g_2$ iff $i_1=i_2$ and $c_{i_1}=c'_{i_1}.$ Hence $g_1\cong g_2$ iff $\lim_{x\to\infty}(g_1(x)/g_2(x))=1.$ 

This last bi-conditional statement is yet again false without the assumption $g_1, g_2\in \mathcal{C}(g),$ as in the aforementioned explicit example, $\lim_{x\to\infty}(g_1(x)/g_2(x))=1$ but $g_1\ncong g_2.$ 
\end{remark*}

\begin{lemma}\label{L:L2}
Let $g\in \CT$ and  $g_1, g_2\in \CC(g)$ with $g_2\ll g_1$ and $g_1(x)\nrightarrow 0$ as $x\to \infty.$
\begin{enumerate}
\item If $g_1\ncong g_2,$ then $S_h g_1-g_2\sim g_1$ for all $h\in \mathbb{R}.$

\item If $g_1\cong g_2,$ then $S_h g_1-g_2\ll g_1'$ for all $h\in \mathbb{R},$ and $S_h g_1-g_2\sim g_1'$ for all $h\in \mathbb{R}\setminus\{0\}.$
\end{enumerate}
\end{lemma}

\begin{proof}
Let $g\in \CT_i.$ If $g_1\sim g^{(i_1)}$ for some $0\leq i_1\leq i,$ then, by Lemma~\ref{L:-1}, and the comment right before Lemma~\ref{L:1}, we have that both $g_1, S_h g_1\in \mathcal{T}_{i-i_1}.$ 


(i) If $g_2(x)\to 0$ as $x\to\infty,$ then the result follows by Lemma~\ref{L:BK}. Otherwise, let \[g_1(x)=\sum_{j=i_1}^{i} c_j g^{(j)}(x)+e_1(x),\;\;\text{and}\;\;g_2(x)=\sum_{j=i_2}^{i} c'_j g^{(j)}(x)+e_2(x)\] be the $g$-forms of $g_1, g_2.$ Since $ g_1\ncong g_2,$ we get that $\lim_{x\to\infty}(g_2(x)/g_1(x))=0$ if $i_2>i_1$ and $\lim_{x\to\infty}(g_2(x)/g_1(x))=c'_{i_1}/c_{i_1}\neq 0,$ if $i_1=i_2,$ so, in any case $S_h g_1-g_2\sim g_1.$

(ii) Since $g_1\cong g_2$ we get that $i_1=i_2$ and $c_{i_1}=c'_{i_1}.$ So $$S_h g_1-g_2=c_{i_1}(S_h g^{(i_1)}-g^{(i_1)})+\sum_{j=i_1+1}^i (c_{j}S_h g^{(j)}-c'_{j}g^{(j)})+S_h e_1-e_2.$$
Using Lemma~\ref{L:1}, we get that for any $h\neq 0,$ we have $$S_h g_1-g_2\sim S_h g^{(i_1)}-g^{(i_1)}\sim g^{(i_1+1)}\sim g'_1.$$ Finally, for $h=0,$ we have that $$g_1-g_2\ll g^{(i_1+1)}\sim g'_1.$$ The proof is now complete.
\end{proof}

\medskip

In the following, we assume without loss of generality that a nice family $(\CA_1,\ldots,\CA_\ell)$ contains no $\ell$-tuples of bounded functions (equivalently, of functions which converge to $0$)  for otherwise we can remove them (with the use of Cauchy-Schwarz inequality). 

\medskip

We will now define the weight (i.e., complexity) of a nice family.

\begin{definition*}
Let $g_1\in \mathcal{T}_d,$ for some $d\in \omega.$ We define \[\CA_1'=\{g_{1,j}\in \CA_1:\;|g_{1,j}(x)|\to\infty\;\;\text{as}\;\;x\to\infty\};  \]
and, for $2\leq i\leq \ell,$ 
\[\CA_i'=\{g_{i,j}\in \CA_i:\;|g_{i,j}(x)|\to\infty\;\;\text{and}\;\;g_{i',j}(x)\to 0\;\;\text{as}\;\;x\to\infty, \;\;\text{for}\;\;i'<i\}. \]
For $1\leq i\leq \ell$ and $0\leq j\leq d,$ let $w_{i,j}$ be the number of non-equivalent distinct classes of functions from $\mathcal{T}_j$ in $\CA_i'.$ The \emph{weight} of $(\CA_1,\ldots,\CA_\ell)$ is defined to be the matrix $(w_{i,d-j})_{1\leq i\leq \ell, 0\leq j\leq d}$. We order the weights, after viewing the matrix $(w_{i,d-j})$ as the vector $(w_{1,d},\ldots,w_{1,0},\ldots,w_{\ell,0}),$ lexicographically.\footnote{ The weight can be defined for any family of functions from $\mathcal{T},$ but we will only consider nice ones.}
\end{definition*}


\medskip

Analogously to \cite[Lemma~5.3]{F3}, we have that every strictly decreasing sequence of weights of nice families of $\ell$-tuples of functions from $\mathcal{T}$ is finite and eventually terminates at the zero vector.

\medskip

Following the proof of \cite[Lemma~5.5]{F3}, using the lemmas that we proved above, we will now show that not only the niceness notion is preserved via the vdC-operations but also under specific ones we achieve reduction of the complexity of the family.

\begin{lemma}\label{L:vdC red}
Let $(\CA_1,\ldots,\CA_\ell)$ be a nice family of $\ell$-tuples of functions, and suppose that $\deg(g_{1,1})\geq 1$. Then there exists $(\tilde{g}_1,\dots,\tilde{g}_\ell)\in \CA_1\cup\{0\}\times\cdots\times \CA_\ell\cup\{0\}$ such that for every $h\in\N,$ the family $(\tilde{g}_1,\ldots,\tilde{g}_\ell,h)\text{-}\emph{vdC}(\CA_1,\ldots,\CA_\ell)$ is nice with weight strictly smaller than that of $(\CA_1,\ldots,\CA_\ell).$ 
\end{lemma}

\begin{proof}  Let $1\leq i\leq \ell$ be the largest integer such that $\CA'_i\neq \emptyset.$ If $i\neq 1,$ we take $\tilde{g}_1=0,\ldots,\tilde{g}_{i-1}=0$ and $\tilde{g}_i$ a function in $\CA'_i$ (this defines the rest of the functions in the $\ell$-tuple). Then, via the vdC-operation, for any $h\in\mathbb{N},$ the first $(i-1)(d+1)$ coordinates of the vector remain unchanged, while there is a reduction in the next $d+1$.

If $i=1$ and $\CA_1$ is a singleton, we choose $(\tilde{g}_1,\dots,\tilde{g}_\ell)=(g_{1,1},\ldots,g_{\ell,1}).$ The reduction of the complexity follows now from Lemma~\ref{L:1}. If $\CA_1$ has more than one elements and $g\cong g_{1,1}$ for all $g\in\CA_1,$ we choose $(\tilde{g}_1,\dots,\tilde{g}_\ell)=(g_{1,1},\ldots,g_{\ell,1});$ otherwise, we choose $(\tilde{g}_1,\dots,\tilde{g}_\ell)$ with $\tilde{g}_1\ncong g_{1,1}$ with minimal degree in $\CA'_1.$ Using Lemmas~\ref{L:1} and ~\ref{L:L2} we get a reduced (because of the first $d+1$ coordinates) weight vector.

\medskip

We will now show that, for all $h\in\mathbb{N},$ the family $(\tilde{g}_1,\dots,\tilde{g}_\ell,h)$-vdC$(\CA_1,\ldots,\CA_\ell)$ is nice.\\

\noindent {\bf{Claim 1.}} The (i) property of a nice family holds for all $h\in\mathbb{N}.$ 

Indeed, we will show that for any $h\in\mathbb{N}$ we have:

(a) $S_h g_{1,1}-S_h g_{1,j}\to\infty$ for $2\leq j\leq m;$

(b) $S_h g_{1,1}-g_{1,j}\to\infty$ for $1\leq j\leq m;$

(c) $S_h g_{1,j}-\tilde{g_1}\ll S_h g_{1,1}-\tilde{g}_1$ for $2\leq j\leq m;$

(d) $g_{1,j}-\tilde{g}_1\ll S_h g_{1,1}-\tilde{g}_1$ for $1\leq j\leq m.$

\medskip

(a): It follows immediately since $g_{1,1}-g_{1,j}\to\infty$ for all $2\leq j\leq m.$

(b): If $g_{1,1}\ncong g_{1,j},$ we have that $S_h g_{1,1}-g_{1,j}\sim g_{1,1},$ while if $g_{1,1}\cong g_{1,j},$ we have that $S_h g_{1,1}-g_{1,j}\sim g'_{1,1}.$ The result follows in both cases by the growth condition of $g_{1,1}.$

(c) and (d): If $g_{1,1}\ncong \tilde{g}_{1},$ then $S_h g_{1,1}-\tilde{g}_1\sim g_{1,1}$ and the result follows since $g_{1,j}\ll g_{1,1}.$ If $g_{1,1}\cong \tilde{g}_{1},$ then by the construction, $g_{1,j}\cong \tilde{g}_{1}$ for all $1\leq j\leq m,$ so, $S_h g_{1,j}-\tilde{g}_1\sim g'_{1,j}$ and $g_{1,j}-\tilde{g}_1\ll g'_{1,j},$ hence the result follows since $g'_{1,j}/g'_{1,1}\to 1.$

\medskip

\noindent {\bf{Claim 2.}} Property (ii) of a nice family holds for all $h\in\mathbb{N}.$

It suffices to show that for $h\in\mathbb{N}$ we have:

(a) $S_h g_{i,j}-\tilde{g}_i\prec S_h g_{1,1}-\tilde{g}_1,$ for $2\leq i\leq \ell,$ $1\leq j\leq m;$ and 

(b) $g_{i,j}-\tilde{g}_i\prec S_h g_{1,1}-\tilde{g}_1,$ for $2\leq i\leq \ell,$ $1\leq j\leq m.$

\medskip

If $g_{1,1}\ncong \tilde{g}_{1},$ then $S_h g_{1,1}-\tilde{g}_1\sim g_{1,1}$ and the result follows for both (a) and (b) since $g_{i,j}\prec g_{1,1}.$

If $g_{1,1}\cong \tilde{g}_{1},$ we have $g_{1,j}\cong \tilde{g}_{1}$ for all $1\leq j\leq m,$ and $\tilde{g}_{i}=g_{i,1}$ for all $1\leq i\leq \ell.$ Then $S_h g_{1,1}-\tilde{g}_1\sim g'_{1,1}$ and (b) follows since $g_{i,j}-\tilde{g}_i=g_{i,j}-g_{i,1}\prec g_{1,1}-g_{1,j}\ll g'_{1,1}.$ To show (a), write, $S_h g_{i,j}-\tilde{g}_i=(S_h g_{i,j}-g_{i,j}) + (g_{i,j}-\tilde{g}_i).$ Since $g_{i,j}\prec g_{1,1},$ we get that $S_h g_{i,j}-g_{i,j}\sim g'_{i,j}\prec g'_{1,1},$ and also, by (b), that  $g_{i,j}-\tilde{g}_i\prec g'_{1,1},$ from which the result follows.

\medskip

\noindent {\bf{Claim 3.}} Property (iii) of a nice family holds for all $h\in\mathbb{N}.$

We will show that for any $h\in\mathbb{N}$ we have

(a) $S_h g_{i,1}-S_h g_{i,j}\prec S_h g_{1,1}-S_h g_{1,j},$ for $2\leq i\leq \ell,$ $2\leq j\leq m;$ and 

(b) $S_h g_{i,1}-g_{i,j}\prec S_h g_{1,1}-g_{1,j},$ for $2\leq i\leq \ell,$ $1\leq j\leq m.$

\medskip

(a): It follows immediately from the definition of the nice family.

(b): If $g_{1,1}\ncong g_{1,j},$ then $S_h g_{1,1}-g_{1,j}\sim g_{1,1}$ and the result follows from the fact that $g_{i,j}\prec g_{1,1}$ for all $2\leq i\leq \ell,$ $1\leq j\leq m.$ If $g_{1,1}\cong g_{1,j},$ then $S_h g_{1,1}-g_{1,j}\sim g'_{1,1}.$ Write $S_h g_{i,1}-g_{i,j}= (S_h g_{i,1}-g_{i,1}) + (g_{i,1}-g_{i,j})$ and note that $S_h g_{i,1}-g_{i,1}\sim g'_{i,1}\prec g'_{1,1}$ and $ g_{i,1}-g_{i,j}\prec g_{1,1}-g_{1,j}\ll g'_{1,1}$ to obtain the result.
\end{proof}

\section{The weakly mixing case}\label{wm}

\medskip

In this section we show Proposition~\ref{P:1} which implies the case where our commuting transformations are weakly mixing (Corollary~\ref{cor:wm}). 
In order to do so, we follow the approach of \cite{F3}. More specifically, using some intermediate results, we show the base case, which corresponds to sublinear iterates, and then finish the proof with (PET) induction, using Lemma~\ref{L:vdC red}.

\medskip

The following result informs us that an average with linear iterates can be bounded by a single Host-Kra seminorm.

\begin{lemma}[\cite{F3}]\label{L:2}
For $m\in \mathbb{N}$ let $(X,\mathcal{B},\mu,T)$ be a system, $f_1,\ldots,f_m\in L^\infty(\mu)$ functions bounded by $1,$ and $\alpha_1,\ldots,\alpha_m$ non-zero real numbers with $\alpha_i\neq \alpha_1$ for all $2\leq i\leq m.$ Then there exists $C\equiv C(m,\alpha_2,\ldots,\alpha_m)$ such that $$\limsup_{N-M\to\infty}\sup_{\norm{f_2}_\infty,\ldots,\norm{f_m}_\infty\leq 1}\norm{\frac{1}{N-M}\sum_{n=M}^{N-1} T^{[\alpha_1 n]}f_1\cdots T^{[\alpha_m n]}f_m}_2\leq C\nnorm{f_1}_{2m,T}.$$
\end{lemma}



Differences of integer parts lead to (bounded) error terms, i.e., terms which take finitely many values. By passing to a product system, we can bound averages 
with such error terms by averages 
where the error terms are fixed. 

\begin{lemma}[\cite{F3}]\label{L:4}
For $\ell, m\in \mathbb{N}$ let $(X,\mathcal{B},\mu,T_1,\ldots,T_\ell)$ be a system, $f_1,\ldots,f_m\in L^\infty(\mu),$ and, for $1\leq i\leq m,$ $1\leq j\leq \ell,$ $(a_{i,j}(n))_n$ sequences with integer values and $(e_{i,j}(n))_n$ sequences taking values in some finite set of integers $F.$ Then for any $N\in\mathbb{N}$ we have $$\sup_{E\subseteq \mathbb{N}}\norm{\frac{1}{N}\sum_{n=1}^N \prod_{i=1}^m (T_1^{a_{i,1}(n)+e_{i,1}(n)}\cdots T_\ell^{a_{i,\ell}(n)+e_{i,\ell}(n)})f_i\cdot {\bf{1}}_E(n)}^2_{L^2(\mu)}$$ $$\leq |F|^{2\ell m}\cdot \max_{c_{i,j}\in F} \norm{\frac{1}{N}\sum_{n=1}^N \prod_{i=1}^m (\tilde{T}_1^{a_{i,1}(n)+c_{i,1}}\cdots \tilde{T}_\ell^{a_{i,\ell}(n)+c_{i,\ell}})\tilde{f}_i}_{L^2(\tilde{\mu})},$$ where $\tilde{T}=T\times T,$ $\tilde{\mu}=\mu\times\mu$ and $\tilde{f}=f\otimes \bar{f}.$ 
\end{lemma}



 

We are now ready to prove Proposition~\ref{P:1}, a crucial step towards the proof of Theorem~\ref{T:main_general}.

\begin{proof}[Proof of Proposition~\ref{P:1}]
The base case is when $g_{1,1}\in \CT_0.$ We will show that if $\nnorm{f_1}_{2m+1,T_1}=0,$ then we have the result.

If $g_1\in \CT_d,$ we have that $g_{1,1}\sim g_1^{(d)},$ so, for $1\leq j\leq \ell,$ we can write (by Lemma~\ref{L:-1}) $g_{1,j}=\lambda_{1,j}g_1^{(d)}+\tilde{e}_{1,j}$ where $\tilde{e}_{1,j}(x)\to 0$ as $x\to\infty$  and $\lambda_{1,1}\neq 0.$ Note that $g_{1,1}\nrightarrow 0$ because of property (i) of the nice family. So, for any $2\leq j\leq m,$ if $\alpha_{j}:=\lambda_{1,j}/\lambda_{1,1}\in \mathbb{R}$ and $c_j:=\tilde{e}_{1,j}-\alpha_j\tilde{e}_{1,1},$ then we have $g_{1,j}=\alpha_j g_{1,1}+c_j.$ For $2\leq i\leq \ell,$ $2\leq j\leq m,$ let 
$$\tilde{g}_{i,j}:=g_{i,j}\circ g^{-1}_{1,1}.$$ 
 Note that $\tilde{g}_{i,j}(x)\prec x$ 
  for all $2\leq i\leq \ell,$ $2\leq j\leq m.$ Also, for all $2\leq i\leq \ell,$ $2\leq j\leq m,$ we can write
   $[g_{1,j}(n)]=[\alpha_j[g_{1,1}(n)]]$
   $+e_{1,j}(n),$ and $[g_{i,j}(n)]=[\tilde{g}_{i,j}([g_{1,1}(n)])]+e_{i,j}(n),$ where $(e_{i,j}(n))_n$ are sequences of integers taking finitely many values. Note also that if some $g_{i,j}(x)\to 0,$ as $x\to\infty,$  then we have $[g_{i,j}(n)]=e_{i,j}(n),$ so, we can assume without loss of generality that every $g_{i,j}$ is in $\CT_0.$ 
   
   Using Lemma~\ref{L:4}, it suffices to show that $$\lim_{N\to\infty}\norm{\frac{1}{N}\sum_{n=1}^N (\tilde{T}_1^{[g_{1,1}(n)]}\prod_{i=2}^\ell \tilde{T}_i^{[\tilde{g}_{i,1}([g_{1,1}(n)])]})\tilde{f}_1\cdot\prod_{j=2}^m (\tilde{T}_1^{[\alpha_j[g_{1,1}(n)]]}\prod_{i=2}^\ell \tilde{T}_i^{[\tilde{g}_{i,j}([g_{1,1}(n)])]})\tilde{f}_j}_{L^2(\tilde{\mu})}=0,$$
    where $\tilde{T}_i=T_i\times T_i,$ $\tilde{\mu}=\mu\times\mu$ and $\tilde{f}_j=\prod_{i=1}^\ell\tilde{T}_i^{\rho_{i,j}}(f_j\otimes \bar{f}_j)$ ($\rho_{i,j}$ are the constants that we get from Lemma~\ref{L:4} to obtain the maximum value), $1\leq i\leq \ell,$ $1\leq j\leq m.$
   
According to Lemma~\ref{L:sub-linear} (for $\ell=1,$ using $\mathcal{T}_0\subseteq \mathcal{F}$) it suffices to show that $$\lim_{N\to\infty}\norm{\frac{1}{|I_N|}\sum_{n\in I_N}(\tilde{T}_1^{n}\prod_{i=2}^\ell \tilde{T}_i^{[\tilde{g}_{i,1}(n)]})\tilde{f}_1\cdot\prod_{j=2}^m (\tilde{T}_1^{[\alpha_j n]}\prod_{i=2}^\ell \tilde{T}_i^{[\tilde{g}_{i,j}(n)]})\tilde{f}_j}_{L^2(\tilde{\mu})}=0,$$ where $(I_n)_n$ is a sequence of intervals with lengths increasing to infinity. 

$g_{i,j}$'s, as functions in $\CT_0,$ are eventually monotone, hence $\tilde{g}_{i,j}$ are eventually monotone, and since $\tilde{g}_{i,j}\prec x,$ we have that $\tilde{g}_{i,j}(x+1)-\tilde{g}_{i,j}(x)$ converge to $0$ and have eventually constant sign. So, we can assume without loss of generality that the sequences $([\tilde{g}_{i,j}(n)])_n$ are constant in each interval $I_N,$ since we can decompose $I_N$ (except a finite set of fixed cardinality) into sub-intervals of lengths tending to infinity in a way that $([\tilde{g}_{i,j}(n)])_n$ are fixed in each sub-interval. So, using Lemma~\ref{L:2} (for every $N$ pick the $j$-th function to be $\prod_{i=2}^\ell \tilde{T}_i^{[\tilde{g}_{i,j}(n)]}\tilde{f}_j$) and the fact that $\nnorm{f_1}_{2m+1,T_1}=0,$ we have, by the properties of the seminorms, that $\nnorm{\tilde{f}_1}_{2m,\tilde{T}_1}=0,$ completing the base case.
\medskip

For the inductive step, let $(\mathcal{A}_1,\ldots,\mathcal{A}_\ell)$ be a nice family of $m$ $\ell$-tuples of functions with weight $W,$ where $g_{1,1}\in \mathcal{T}_i$ for some $i\geq 1,$ and assume that the statement is true for nice families of $2m$ $\ell$-tuples of functions with weight $W'<W$ with $k(W',2m)$ being the integer for which the conclusion holds. We will show that $k=\max_{W'<W}k(W',2m)+1$ is the required $k$ and we will complete the proof.    

Assuming that $\nnorm{f_1}_{k,T_1}=0$, via Lemma~\ref{vdClemma}, it suffices to show that, for sufficiently large $h,$ we have that the averages of
\[\int \prod_{j=1}^m \left(\prod_{i=1}^\ell T_i^{[g_{i,j}(n+h)]}\right)f_j\cdot \prod_{j=1}^m \left(\prod_{i=1}^\ell T_i^{[g_{i,j}(n)]}\right)\bar{f}_j\;d\mu\]
converge to $0$ as $N\to\infty.$ If $(\tilde{g}_1,\ldots,\tilde{g}_\ell)$ denotes the $\ell$-tuple guaranteed by Lemma~\ref{L:vdC red}, it suffices to show that the averages of 
\[\prod_{j=1}^m \left(\prod_{i=1}^\ell T_i^{[g_{i,j}(n+h)-\tilde{g}_i(n)]+e_{i,j}(n)}\right)f_j\cdot \prod_{j=1}^m \left(\prod_{i=1}^\ell T_i^{[g_{i,j}(n)-\tilde{g}_i(n)]+e_{i+\ell,j}(n)}\right)\bar{f}_j\] converge to $0$ in $L^2(\mu),$ where $e_{i,j}(n),$ $1\leq i\leq 2\ell,$ $1\leq j\leq m$ take values in $\{0,1\}.$
Using Lemma~\ref{L:4}, we have to show that the average
\[\prod_{j=1}^m \left(\prod_{i=1}^\ell \tilde{T}_i^{[g_{i,j}(n+h)-\tilde{g}_i(n)]}\right)\tilde{f}_j\cdot \prod_{j=1}^m \left(\prod_{i=1}^\ell \tilde{T}_i^{[g_{i,j}(n)-\tilde{g}_i(n)]}\right)\tilde{f}_{j+m}\] converges to $0$ in $L^2(\tilde{\mu})$ as $N\to\infty,$ where $\tilde{T}_i=T_i\times T_i,$ $\tilde{\mu}=\mu\times \mu,$ $\tilde{f}_j=\prod_{i=1}^\ell \tilde{T}^{\rho_{i,j}}_i(f_j\otimes \bar{f}_j)$ and $\tilde{f}_{j+m}=\prod_{i=1}^\ell \tilde{T}^{\rho_{i+\ell,j}}_i (\bar{f}_j\otimes f_j)$ (as in the base case, $\rho_{i,j},$ $1\leq i\leq 2\ell,$ $1\leq j\leq m$ are the constants that we get from Lemma~\ref{L:4}). Note that in this procedure we remove any term whose iterate is bounded as they do not contribute on the convergence to $0$.

By Lemma~\ref{L:vdC red}, we have that the family $(\tilde{g}_1,\ldots,\tilde{g}_\ell,h)-\text{vdC}(\mathcal{A}_1,\ldots,\mathcal{A}_\ell)$ is nice with weight $W'<W$. Its first iterate, under the transformations $\tilde{T}_i$'s, is the $\ell$-tuple $([g_{1,1}(n+h)-\tilde{g}_1(n)],\ldots,[g_{\ell,1}(n+h)-\tilde{g}_\ell(n)])$ and it is applied to the function $\tilde{f}_1.$ The claim now follows by induction, since $\nnorm{f_1}_{k,T_1}=0$ implies $\nnorm{\tilde{f}_1}_{k(W',2m),\tilde{T}_1}=\nnorm{f_1\otimes \bar{f}_1}_{k(W',2m),\tilde{T}_1}=0.$
\end{proof}

Immediate implication of the previous result for the nice family $\mathcal{A}_1=(g_1,0,\ldots,0),$ $\ldots,$ $\mathcal{A}_\ell=(0,\ldots,0,g_\ell)$ is the following:

\begin{proposition}\label{P:wm_standard}
For $\ell \in \mathbb{N}$ let $(X,\mathcal{B},\mu,T_1,\ldots,T_\ell)$ be a system, $f_1,\ldots,f_\ell\in L^\infty(\mu),$ and $g_\ell\prec\ldots\prec g_1\in \CT$ with $g_1\in \CT_d,$ $d\in \omega.$ There exists $k\equiv k(d,\ell)\in\mathbb{N}$ such that if $\nnorm{f_1}_{k,T_1}=0,$ then  \[\lim_{N\to\infty}\norm{\frac{1}{N}\sum_{n=1}^N T_1^{[g_1(n)]}f_1\cdots T_\ell^{[g_\ell(n)]}f_\ell}_2=0.\] 
\end{proposition}


Another implication is Corollary~\ref{cor:wm}, i.e., for weakly mixing transformations and for iterates of the form $a_i=[g_i],$ where $g_i\in\mathcal{T}$ are of different growth, the limit of \eqref{E:Multiple} is the expected one. This follows by a standard classical argument (which we omit), using the fact that every $T_i$ as weakly mixing is also ergodic, hence, we have $\mathbb{E}(f_i|\mathcal{I}(T_i))=\int f_i\;d\mu.$

\begin{remark*}
As it was highlighted in Section~\ref{results}, Corolarry~\ref{cor:wm} covers some additional to \cite{BK} cases. Consider for example the pair of functions from $\mathcal{T}_0$ of different growths:
\[\left\{g_1(x)=\sqrt{x}\log x, g_2(x)=\sqrt{x}(2+\cos\sqrt{\log x})\right\}.\] While \cite[Theorem~B]{BK} cannot be applied,\footnote{ The functions $g_1, g_2$ don't satisfy the $R$-property (see \cite{BK}).} for $\ell=2$ and $T_1=T_2$ w.m., Corollary~\ref{cor:wm} implies the convergence of the corresponding average \eqref{E:Multiple} to the expected limit.
\end{remark*}

\section{Towards the convergence to the expected limit, \\ equidistribution results on nilmanifolds for several nil-orbits}\label{S:equi}

In the proof of Theorem~\ref{T:main_general}, we will consider separately the case where all the $g_i$'s are sublinear (case that follows from things that have already been discussed), and the cases where all (Case 2) or some (Case 3, which actually follows by similar arguments to Case~2) of them are superlinear (i.e., have positive degree).

The first step towards this direction is to replace the condition of Proposition~\ref{P:wm_standard}, $\nnorm{f_1}_{k,T_1}=0,$ with $\nnorm{f_i}_{k,T_i}=0$ for some $1\leq i\leq \ell.$ Frantzikinakis, in \cite{F3}, overcame this technicality by using dual sequences (see \cite[Subsection~3.3]{F3}) and a weak decomposition result (\cite[Proposition~3.4]{F3}), obtaining correlation estimates (see \cite[Sections~6 and 7]{F3}) for the multiple ergodic averages of interest. As 
one can follow the arguments from \cite{F3} (replacing the results about Hardy field functions with the corresponding ones for functions in $\mathcal{T}$ proved in previous sections), we skip the proofs of the next two results.  

\begin{proposition}$($\emph{Analogous to} \cite[Proposition~7.1]{F3}$)$\label{p: general seminorm}
For $\ell\in \mathbb{N}$ let $(X,\mathcal{B},\mu,T_1,\ldots,T_\ell)$ be a system, $f_1,\ldots,f_\ell\in L^\infty(\mu)$ and $g_\ell\prec\ldots\prec g_1\in \bigcup_{i= 1}^d\mathcal{T}_i,$ $d\in \mathbb{N}.$ There exists $k=k(d,\ell)\in \mathbb{N}$ such that if $\nnorm{f_i}_{k,T_i}=0$ for some $1\leq i\leq \ell,$ then \[\lim_{N\to\infty}\norm{\frac{1}{N}\sum_{n=1}^N T_1^{[g_1(n)]}f_1\cdots T_\ell^{[g_\ell(n)]}f_\ell}_2=0.\]
\end{proposition}

Actually the following variant of the previous result also holds (via the use of the intermediate results of the corresponding expressions composed with products of dual sequences):

\begin{proposition}$($\emph{Analogous to} \cite[Proposition~7.3]{F3}$)$\label{p:general with R}
For $\ell\in \mathbb{N}$ let $(X,\mathcal{B},\mu,T_1,\ldots,T_\ell)$ be a system, $f_1,\ldots,f_\ell\in L^\infty(\mu)$ and $g_\ell\prec\ldots\prec g_1\in \bigcup_{i= 1}^d\mathcal{T}_i,$ $d\in \mathbb{N}.$ There exists $k=k(d,\ell)\in \mathbb{N}$ such that if $\nnorm{f_i}_{k,T_i}=0$ for some $1\leq i\leq \ell,$ then for every $R\in \mathbb{N}$ we have 
\[\limsup_{N\to\infty}\frac{1}{N}\sum_{n=1}^N\norm{\frac{1}{R}\sum_{r=1}^R T_1^{[g_1(Rn+r)]}f_1\cdots T_\ell^{[g_\ell(Rn+r)]}f_\ell}^2_2\leq \frac{1}{R}.\]
\end{proposition}

Next, we follow \cite{F4} to show some equidistribution results on nilmanifolds for several nil-orbits along sequences of tempered functions. Up to this point no additional assumption is needed on the $g_i$'s than the one of different growth rates (and of course that all belong to $\CT$). To obtain the equidistribution results that follow though, we have to make sure that any non-trivial linear combination of $g_i$'s is still a tempered function.







\begin{theorem}\label{T:main_equi}
For $\ell\in \mathbb{N}$ let $g_\ell\prec\ldots\prec g_1\in \bigcup_{i\geq 1}\mathcal{T}_i$ with the property that $\sum_{i=1}^{\ell}\lambda_i g_i\in \CT$ for all $(\lambda_1,\ldots,\lambda_\ell)\in \mathbb{R}^\ell\setminus\{\vec{0}\},$ $X_i=G_i/\Gamma_i$ nilmanifolds, $b_i\in G_i$ and $x_i\in X_i,$ $1\leq i\leq \ell.$  Then the sequence \[(b_1^{[g_1(n)]}x_1,\ldots,b_\ell^{[g_\ell(n)]}x_\ell)_n\] is equidistributed in the nilmanifold $\overline{(b_1^nx_1)}_n\times\cdots\times\overline{(b_\ell^nx_\ell)}_n.$
\end{theorem}

This result is the analogous to \cite[Theorem~1.3~(ii)]{F4} for Hardy field functions. In order to prove it, it suffices, via \cite[Lemma~5.1]{F4}, to show the following (analogous to \cite[Theorem~1.3~(i)]{F4}) result:

\begin{theorem}\label{T:main_equi2}
For $\ell\in \mathbb{N}$ let $g_\ell\prec\ldots\prec g_1\in \bigcup_{i\geq 1}\mathcal{T}_i$  with the property that $\sum_{i=1}^{\ell}\lambda_i g_i\in \CT$ for all $(\lambda_1,\ldots,\lambda_\ell)\in \mathbb{R}^\ell\setminus\{\vec{0}\},$ $X_i=G_i/\Gamma_i$ nilmanifolds, with $G_i$ connected and simply connected, $b_i\in G_i$ and $x_i\in X_i,$ $1\leq i\leq \ell.$ Then the sequence \[(b_1^{g_1(n)}x_1,\ldots,b_\ell^{g_\ell(n)}x_\ell)_n\] is equidistributed in the nilmanifold $\overline{(b_1^sx_1)}_{s\in \mathbb{R}}\times\cdots\times\overline{(b_\ell^sx_\ell)}_{s\in \mathbb{R}}.$
\end{theorem}

\cite[Lemma~5.2]{F4} (working on $X=G/\Gamma,$ where $X=X_1\times\cdots\times X_\ell,$ $G=G_1\times\cdots\times G_\ell$ product of connected and simply connected $G_i$'s, and $\Gamma=\Gamma_1\times\cdots\times \Gamma_\ell$) implies that for $\ell\in\mathbb{N},$ nilmanifolds $X_i=G_i/\Gamma_i,$ and $b_i \in G_i,$ $1\leq i\leq \ell,$ we can find an $s_0\in \mathbb{R}$ such that $b_i^{s_0}$ acts ergodically on the nilmanifold $\overline{(b_i^s\Gamma_i)}_{s\in \mathbb{R}}.$ Using this, we have that Theorem~\ref{T:main_equi2} follows by the following proposition (analogous to \cite[Proposition~5.3]{F4}), the sketch of proof of which reveals why we have to postulate the assumption about linear combinations.

\begin{proposition}\label{P:main_qaui}
For $\ell\in \mathbb{N}$ let $g_\ell\prec\ldots\prec g_1\in \bigcup_{i\geq 1}\mathcal{T}_i$ with the property that $\sum_{i=1}^{\ell}\lambda_i g_i\in \CT$ for all $(\lambda_1,\ldots,\lambda_\ell)\in \mathbb{R}^\ell\setminus\{\vec{0}\},$ $X_i=G_i/\Gamma_i$ nilmanifolds, with $G_i$ connected and simply connected,  and $b_i\in G_i$ acting ergodically on $X_i,$ $1\leq i\leq \ell.$ Then the sequence \[(b_1^{g_1(n)}\Gamma_1,\ldots,b_\ell^{g_\ell(n)}\Gamma_\ell)_n\] is equidistributed in the nilmanifold $X_1\times\cdots\times X_\ell.$
\end{proposition}

\begin{proof}[Sketch of proof] Following \cite{F4} we present the main arguments, showing the corresponding intermediate steps, which cover our case.  

We can assume that $X_i=X,$ $1\leq i\leq \ell,$ as the general case is similar. It suffices to show that 
\begin{equation}\label{E:equi53}
\frac{1}{N}\sum_{n=1}^N F(b_1^{g_1(n)}\Gamma,\ldots, b_\ell^{g_\ell(n)}\Gamma)
\end{equation}
goes to $0$ as $N\to\infty,$ where $F$ is a continuous function on $X^\ell$ with $0$ integral.

Assuming that each $g_i\in \mathcal{T}_{k_i}$ we have that $g_i^{(k_i+1)}(x)\to 0$ as $x\to\infty.$ Fixing an $R\in \mathbb{N},$ using the Taylor expansion for each $g_i,$ $1\leq i\leq \ell,$ for $1\leq r\leq R,$ we can write 
\[ g_i(Rn+r)=p_{i,R,n}(r)+e_{R}(n),\;\;\text{where}\;\; p_{i,R,n}(r)=\sum_{j=0}^{k_i}\frac{r^j}{j!}g_i^{(j)}(Rn)\;\;\text{and}\;\;e_R(n)\to 0.\] 
Using these polynomials, we have that
\begin{equation}\label{E:equi53.2}
\frac{1}{RN}\sum_{n=1}^{RN} F(b_1^{g_1(n)}\Gamma,\ldots,b_\ell^{g_\ell(n)}\Gamma)=\frac{1}{N}\sum_{n=1}^N\frac{1}{R}\sum_{r=1}^R F(b_1^{p_{1,R,n}(r)}\Gamma,\ldots,b_\ell^{p_{\ell,R,n}(r)}\Gamma)+\tilde{e}_{R}(N),
\end{equation} where $\tilde{e}_R(N)\to 0.$
Following the arguments of \cite[Proposition~5.3]{F4} we have to verify that the functions $g_i,$ $1\leq i\leq \ell,$ satisfy some assumptions of the intermediate results that imply the conclusion. More specifically:

 \cite[Proposition~4.2]{F4} can be used since the functions $g_i$ satisfy for some $k\in \mathbb{N}$ that \[|g^{(k+1)}_i(x)|\;\;\text{ is decreasing,}\;\; 1/t^k\prec g_i^{(k)}(x)\prec 1,\;\;\text{ and}\;\;(g_i^{(k+1)}(x))^k\prec (g_i^{(k)}(x))^{k+1}.\]
Indeed, by the definition of $\mathcal{T}_{k_i},$ picking $k=k_i+1,$ we have that \[\lim_{x\to\infty}\frac{xg_i'(x)}{g_i(x)}=\alpha_i,\] for some $k_i<\alpha_i\leq k_i+1,$ so 
\[\lim_{x\to\infty}\frac{xg_i^{(k_i+3)}(x)}{g_i^{(k_i+2)}(x)}=\alpha_i-k_i-2<0. \] Using also the fact that $g_i^{(k_i+1)}(x)\to 0$ monotonically as $x\to \infty,$ we have that $g_i^{(k_i+3)}(x)$ has (eventually) 
the opposite sign of $g_i^{(k_i+2)}(x)$, so $|g_i^{(k+1)}(x)|$ is decreasing.

 For the second part note that 
 \[\lim_{x\to\infty}\frac{|g^{k}(x)|}{1/x^k}=\lim_{x\to\infty} x^{k_i}\cdot x|g_i^{(k_i+1)}(x)|=\infty,\] and that $g_i^{k}(x)=g_i^{k_i+1}(x)\prec 1$ by definition.
 
 For the third part, by easy calculations, we have 
 \[\lim_{x\to\infty}\frac{(g_i^{(k+1)}(x))^{k}}{(g_i^{k}(x))^{k+1}}=\lim_{x\to\infty}\left(\frac{xg_i^{(k_i+2)}(x)}{g_i^{(k_i+1)}(x)}\right)^{k_i+1}\cdot\lim_{x\to\infty}\frac{1}{x^{k_i+1}g_i^{(k_i+1)}(x)}=0.\]
 \cite[Lemma~2.1]{F4} can be used since every $g_i$ satisfies 
 \[\frac{g_i(x)}{x(\log x)^2}\prec g_i'(x)\ll \frac{g_i(x)}{x}. \]
 Indeed, \[\frac{g_i(x)}{x(\log x)^2 g_i'(x)}=\frac{g_i(x)}{xg_i'(x)}\cdot\frac{1}{(\log x)^2}\to 0,\;\;\text{ as}\;\;x\to\infty,\;\;\text{ and}\;\;\frac{xg_i'(x)}{g_i(x)}\;\;\text{ is bounded}.\]
 The last two facts that we have to check is that a non-trivial linear combination of $k$-th derivatives of functions which belong to the same $\mathcal{T}_{k}$ and have different growth rates, is a function in $\mathcal{T}_{0},$ and whenever $g$ is a Fej\'{e}r function that $(g(n))_n$ is equidistributed in $\T$.
 
The first fact follows by IV (iii), and the second one by I (iv).

The result now follows by the proof of \cite[Proposition~5.3]{F4}, eventually deducing that 
\[\lim_{R\to\infty}\limsup_{N\to\infty}\frac{1}{N}\sum_{n=1}^N \left| \frac{1}{R}\sum_{r=1}^R F(b_1^{p_{1,R,n}(r)}\Gamma,\ldots,b_\ell^{p_{\ell,R,n}(r)}\Gamma) \right|=0,\] hence, via \eqref{E:equi53.2} we have that \eqref{E:equi53} goes to $0$ as $N\to\infty,$ finishing the argument.
\end{proof}

\begin{remark*}  Following the notation of Theorem~\ref{T:main_equi}, its conclusion implies that for all $F_i\in C(X_i)$ we have
\begin{equation}\label{E:equi_1}
\lim_{N\to\infty}\frac{1}{N}\sum_{n=1}^N\prod_{i=1}^\ell F_i(b_i^{[g_i(n)]}x_i)=\prod_{i=1}^\ell \lim_{N\to\infty}\frac{1}{N}\sum_{n=1}^N F_i(b_i^n x_i).
\end{equation}

Also, under the assumptions of Theorem~\ref{T:main_equi}, the proof of Proposition~\ref{P:main_qaui} (after reducing to the case where each $G_i$ is connected and simply connected and each $b_i\in G_i$ acts ergodicaly on $X_i$), implies that for every $F\in C(X_1\times\cdots\times X_\ell)$ we have that
\begin{equation}\label{E:equi_2}
\lim_{R\to\infty}\limsup_{N\to\infty}\frac{1}{N}\sum_{n=1}^N\left|\frac{1}{R}\sum_{r=1}^R F(b_1^{[g_1(Rn+r)]}x_1,\ldots,b_\ell^{[g_\ell(Rn+r)]}x_\ell)-\int F\;dm_{Y}\right|=0,
\end{equation}
where $m_Y$ is the Haar measure on the subnilmanifold $Y=\overline{(b_1^nx_1)}_n\times\cdots\times\overline{(b_\ell^nx_\ell)}_n.$
\end{remark*}

\section{Proof of main result}\label{main}






We are now in position, following \cite[Subsections~7.3 and 7.4]{F3}, to combine the results from the previous section to show Theorem~\ref{T:main_general}. The proof also uses the following strong decomposition theorem:

\begin{theorem}$($\emph{Strong decomposition}, \cite[Theorem~3.5]{F3}$)$\label{t: strong decomposition}
Let $(X,\mathcal{B},\mu,T)$ be a system, $f\in L^{\infty}(\mu)$ and $k\in\mathbb{N}.$ Then for every $\varepsilon>0,$ there exist functions $f_s,$ $f_u,$ and $f_e,$ with $L^{\infty}(\mu)$ norm at most $2\norm{f}_\infty,$ such that $f=f_s+f_u+f_e,$ $\nnorm{f_u}_{k+1,T}=0,$ $\norm{f_e}_2\leq \varepsilon,$ and for almost every $x\in X$ the sequence $(T^n f_s(x))_n$ is a $k$-step nilsequence.
\end{theorem}

\begin{proof}[Proof of Theorem~\ref{T:main_general}]
Assuming that some $\mathbb{E}(f_i|\mathcal{I}(T_i))=0,$ it suffices to show that 
\begin{equation}\label{E:final}
\limsup_{N\to\infty}\norm{\frac{1}{N}\sum_{n=1}^N\prod_{i=1}^\ell T_i^{[g_i(n)]}f_i}_2=0.
\end{equation}
 We split the proof into three steps.

\medskip

\noindent {\bf Claim 1.} The result holds if all the $g_i$'s belong to $\mathcal{T}_0.$

\medskip

\noindent Indeed, this follows from the remark after the proof of Theorem~\ref{T:sub-linear} in Section~\ref{slF}.

\medskip

\noindent {\bf Claim 2.} The result follows if all the $g_i$'s belong to $\bigcup_{i\geq 1}\mathcal{T}_i.$

\medskip

\noindent Indeed, for $\varepsilon>0,$ using the decomposition from Theorem~\ref{t: strong decomposition}, for any $1\leq i\leq \ell,$ we can write $f_i=f_{i,s}+f_{i,u}+f_{i,e},$ where $\nnorm{f_{i,u}}_{k+1,T_i}=0,$ $\norm{f_{i,e}}_2\leq \varepsilon$ and for almost every $x\in X$ the sequence $(T^n f_{i,s}(x))_n$ is a $k$-step nilsequence. It is clear that the contribution of the terms $f_{i,u}$ is negligible because of Proposition~\ref{p: general seminorm}, while of the terms $f_{i,e}$ is bounded by a constant multiple of $\varepsilon$, hence, it suffices to check the behavior of the terms $f_{i,s}.$ Because of \eqref{E:equi_1} we have that 
\[\lim_{N\to\infty}\frac{1}{N}\sum_{n=1}^N \prod_{i=1}^\ell T_i^{[g_i(n)]} (f_{i,s}(x))=\prod_{i=1}^\ell \lim_{N\to\infty}\frac{1}{N}\sum_{n=1}^N T_i^n (f_{i,s}(x))\] which is equal, up to a constant multiple of $\varepsilon,$ to \[\prod_{i=1}^\ell \lim_{N\to\infty}\frac{1}{N}\sum_{n=1}^N T_i^n (f_{i}(x))=\prod_{i=1}^\ell \tilde{f}_i=0.\] Consequently, the same is true in $L^2(\mu)$ as each $f_i$ is bounded, proving the claim.

\medskip

\noindent {\bf Claim 3.} The result follows if for some $1\leq i_0\leq \ell-1$ we have $g_\ell\prec \ldots\prec g_{i_0+1}\in \mathcal{T}_0$ and $g_{i_0}\prec\ldots\prec g_1\in \bigcup_{i\geq 1}\mathcal{T}_i.$

\medskip

\noindent In this last case, note that 
\[\limsup_{N\to\infty}\norm{\frac{1}{N}\sum_{n=1}^N\prod_{i=1}^\ell T_i^{[g_i(n)]}f_i}_2=\limsup_{N\to\infty}\norm{\frac{1}{N}\sum_{n=1}^N\frac{1}{R}\sum_{r=1}^R\prod_{i=1}^\ell T_i^{[g_i(Rn+r)]}f_i}_2.\]
The right-hand side limit, as the functions $g_{i_0+1},\ldots,g_{\ell}$ are sublinear (so their derivatives are going to $0$), is also equal to
\[\limsup_{N\to\infty}\norm{\frac{1}{N}\sum_{n=1}^N\prod_{i=i_0+1}^\ell T_i^{[g_i(Rn)]}f_i\cdot \frac{1}{R}\sum_{r=1}^R\prod_{i=1}^{i_0} T_i^{[g_i(Rn+r)]}f_i}_2,\] since, for every $R\in \mathbb{N},$ for a set of $n$'s of density $1$ we have that $[g_i(Rn+r)]=[g_i(Rn)],$ $1\leq r\leq R,$ $i_0+1\leq i\leq \ell.$ This last limsup is bounded by 
\[ \prod_{i=i_0+1}^\ell \norm{f_i}_\infty\cdot \limsup_{N\to\infty}\frac{1}{N}\sum_{n=1}^N\norm{\frac{1}{R}\sum_{r=1}^R\prod_{i=1}^{i_0}T_i^{[g_i(Rn+r)]}f_i}_2.\]
 Arguing as in Claim~2, using Proposition~\ref{p:general with R} instead of Proposition~\ref{p: general seminorm} and \eqref{E:equi_2} instead of \eqref{E:equi_1}, we have that this last limsup, as $R\to\infty$, goes to $0,$ as was to be shown. 
\end{proof}


\subsection{Closing comments} 
While for special subclasses of Hardy field (\cite{F3}) and tempered functions (as we just saw) of different growth rates we have convergence to the expected limit for general systems, i.e., more general results comparing to the ones for polynomials where we have to postulate additional assumptions either on the system or the transformations,\footnote{ Indeed, the analogous to Theorem~\ref{T:main_general} statement for polynomials of different degrees is false. Even for $\ell=1$ and $p_1(t)=t^2$ the limit is not in general the expected one. Actually, there are no general results for some particular ``nice'' classes of polynomials for which the limit of \eqref{E:Multiple} is known, with a few exceptions being: for linear iterates (\cite{HK99}), for a very special case of degree $2$ (in the lengthy \cite{Au}), and, for single $T$ and ``independent enough'' real polynomials (\cite{KK}).} no Walsh-type result is known for these classes. On the other hand, the results for polynomials are usually ``more uniform'' as one can replace the conventional Ces\`aro averages, i.e., $\lim_{N\to\infty}\frac{1}{N}\sum_{n=1}^N,$ with uniform ones, i.e., $\lim_{N-M\to\infty}\frac{1}{N-M}\sum_{n=M}^{N-1}$ (or even ones along general F{\o}lner sequences),
 obtaining also Khintchine-type recurrence applications for the corresponding expressions; the same is not true for tempered or Hardy field functions for they can be constant at arbitrarily large intervals.
As special cases of \eqref{E:Multiple} are known for (integer part of) polynomial, Hardy field and tempered functions, it is only natural for someone to ask whether we can have a result where we have a combination of iterates coming from the union of these three classes. We close this article with this exact question.

\medskip

\begin{question*}
Is it true that \eqref{E:Multiple} converges to the expected limit for $a_i=[g_i],$ where the $g_i$'s are ``distinct enough'' polynomial, Hardy field or tempered functions?
\end{question*}


\subsection*{Acknowledgements} Thanks go to V.~Bergelson not only for suggesting the interesting topic of tempered functions but also for his constant support and investment of many hours on numerous meetings during the development of this paper. I also thank deeply N.~Frantzikinakis for his constant support and fruitful discussions on the topic. Last, but not least, I want to express my indebtedness to the anonymous referee, the feedback of whom improved the quality of this paper by correcting a number of typos, suggesting also different approaches in some parts (the cleaner presentation of Case 2 of Proposition~\ref{P:vn} is such an example).

\end{document}